\documentclass[12pt]{amsart}

\usepackage[utf8]{inputenc}
\usepackage{amsmath}
\usepackage{amssymb}
\usepackage{amsthm}
\usepackage{graphicx}
\usepackage{tabularx}
\usepackage{tabularray}
\usepackage{subcaption}

\DeclareMathOperator{\sgn}{sgn}
\newcommand\abs[1]{\left\lvert#1\right\rvert}

\newtheorem{theorem}{Theorem}
\newtheorem{lemma}{Lemma}
\newtheorem{definition}{Definition}
\newtheorem{remark}{Remark}
\newtheorem{example}{Example}
\newtheorem{exercise}{Exercise}
\newtheorem{corollary}{Corollary}

\title[On a special class of equidistant sets...]{On a special class of equidistant sets in the Euclidean space}
\author{\'{A}. Nagy}
\address{Institute of Mathematics, University of Debrecen, H-4002 Debrecen, P. O. Box 400, Hungary}
\email{abris.nagy@science.unideb.hu}
\author{M. Ol\'{a}h}
\address{Institute of Mathematics, University of Debrecen, H-4002 Debrecen, P. O. Box 400, Hungary \newline
\indent ELKH-DE Equations, Functions, Curves and their Applications Research Group}
\email{olah.mark@science.unideb.hu}
\author{M. Stoika}
\address{Department of Mathematics and Informatics, Ferenc Rakoczi II Transcarpathian Hungarian \newline
\indent University, 90200 Beregszász, P. O. Box 33, Transcarpathia, Ukraine}
\email{sztojka.miroszlav@kmf.org.ua}
\author{Cs. Vincze}
\address{Institute of Mathematics, University of Debrecen, H-4002 Debrecen, P. O. Box 400, Hungary}
\email{csvincze@science.unideb.hu}

\begin{document}
\begin{abstract}
An equidistant set in the Euclidean space consists of points having equal distances to both members of a given pair of sets, called focal sets. Since there is no effective formula to compute the distance of a point and a set, it is hard to determine the points of an equidistant set in general. Fortunately, we can use computer-assisted methods due to the theorem of M. Ponce and P. Santib\'a\~nez about the (Hausdorff) convergence of the equidistant sets under the convergence of the focal sets \cite{PS}. One of the most effective ways of approximating an equidistant set is based on equidistant sets with finite focal sets. They form one of the most important special cases \cite{VVOFS}, see also \cite{VOL}.

In the paper we investigate equidistant sets that can be given as the graph of a function. They are called equidistant functions.  The conceptual model is presented in \cite{VVOF} and \cite{NOSV}, where one of the focal sets is the horizontal hyperplane through the origin and the other one is the epigraph of a positive-valued, continuous function. In this case the equidistant points form the graph of another function over the hyperplane. Under certain conditions, such as convexity and differentiability, we can give effective parametric expressions of the equidistant points together with the characterization of the corresponding equidistant function. In a general situation, the hyperplane is the first-order (linear) approximation for one of the focal sets to provide a sufficiently simple starting point for some explicit computations (at least locally). A natural idea is to substitute the hyperplane by a circle (sphere) as a second-order (quadratic) approximation for one of the focal sets in more complicated cases. Such a generalization results in a new type of equidistant functions we are going to investigate in the present paper. 

Before considering the special cases in detail, we present some general observations motivated by preparing the case of the circle (sphere):  a necessary and sufficient condition for the existence of equidistant points along the vertical lines, upper/lower equidistant functions, equidistant functions, a necessary and sufficient condition for the existence of the equidistant function, equidistant functions and the minimum operator (a kind of commuting property). Examples will also be given to illustrate how complicated an equidistant set can be even in case of choosing relatively simple focal sets. The pathological examples show that additional restrictions are necessary to have computable classes of equidistant sets. In the paper we investigate the following special case: one of the focal sets is the epigraph of a positive-valued, convex and continuously differentiable function and the focal set under the epigraph is a circle (sphere) centered at the origin. We give effective parametric expressions of the equidistant points and, assuming higher-order differentiability, the corresponding equidistant function will also be characterized as a particular solution of the characterization problem due to M. Ponce and P. Santib\'{a}\~{n}ez \cite{PS}.
\end{abstract}
\maketitle

\section{Introduction: notations and preliminaries}

Let $K\subset \mathbb{R}^n$ be a subset in the Euclidean coordinate space. The distance between a point $x\in \mathbb{R}^n$ and $K$ is measured by the usual infimum formula
$$d(x, K) := \inf  \{ d(x, y) \ | \ y\in K \},$$
where 
$$d(x,y)=|x-y|=\sqrt{(x^1-y^1)^2+\ldots+(x^n - y^n)^2}$$
is the distance induced by the canonical inner product. The distance-measuring function is Lipschitz continuous provided that $K$ is a nonempty, closed subset in $\mathbb{R}^n$. In particular, 
$$|d(x_1, K)-d(x_2, K)|\leq d(x_1, x_2)$$
for any $x_1$, $x_2\in \mathbb{R}^n$, as a simple application of the triangle inequality shows \cite{NOSV}. Let us define the \textbf{equidistant set} of the focal sets $K$ and $L\subset \mathbb{R}^n$ as the set 
$$\{K=L\}:=\{x\in \mathbb{R}^n\ | \ d(x, K)=d(x, L)\}$$
all of whose points have the same distances from both $K$ and $L$. Equidistant sets\footnote{"We find equidistant sets as conventionally
defined frontiers in territorial domain controversies: for instance, the United Nations
Convention on the Law of the Sea (Article 15) establishes that, in absence of any previous agreement, the delimitation of the territorial sea between countries occurs exactly on the median line every point of which is equidistant of the nearest points to each country"; for the citation see \cite{PS}.} can also be considered as a generalization of conics \cite{PS}. They are often called midsets. The systematic investigations have been started by Wilker's and Loveland's fundamental works \cite{Loveland} and 
\cite{Wilker}. The points of an equidistant set are difficult to determine in general because there is no effective formula to compute the distance between a point and a set. According to the theorem of M. Ponce and P. Santib\'a\~nez about the (Hausdorff) convergence of the equidistant sets under the convergence of the focal sets \cite{PS}, special classes of equidistant sets allow us to approximate the equidistant set in more complicated cases. One of the most effective ways of the approximation is based on equidistant sets with finite focal sets. They form one of the most important special cases \cite{VVOFS}, see also \cite{VOL}.

In the paper we investigate equidistant sets that can be given as the graph of a function. They are called equidistant functions.  The conceptual model is presented in \cite{VVOF} and \cite{NOSV}, where one of the focal sets is the horizontal hyperplane through the origin and the other one is the epigraph of a positive-valued, continuous function. In this case the equidistant points form the graph of another function over the hyperplane. Under certain conditions, such as convexity and diffe\-ren\-tiability, we can give effective parametric expressions of the equidistant points together with the characterization of the corresponding equidistant function as well. In a general situation, the hyperplane is the first-order (linear) approximation for one of the focal sets to provide a sufficiently simple starting point for some explicit computations (at least locally). A natural idea is to substitute the hyperplane by a circle (sphere) as a second-order (quadratic) approximation for one of the focal sets in more complicated cases. Such a generalization results in a new type of equidistant functions we are going to investigate in the paper. 

Before considering the special cases in detail, we present some general observations motivated by preparing the case of the circle (sphere). These observations apply to cases that have in common that one of the focal sets is the epigraph of a continuous function $f\colon \mathbb{R}^n\to \mathbb{R}$. The other focal set $K\subset \mathbb{R}^{n+1}$ is closed and disjoint to the epigraph of $f$. We give a necessary and sufficient condition for the existence of equidistant points along the vertical line through any $x\in \mathbb{R}^n$ (Theorem 1). We also have bounds for the last coordinates of the equidistant points. Therefore, we can introduce the so-called upper/lower equidistant functions. It turns out that they are upper/lower semi-continuous functions. If the upper and lower equidistant functions coincide, then we speak about the equidistant function belonging to the focal sets. As a general observation, we prove a necessary and sufficient condition for the existence of the equidistant function under an additional condition of convexity for $K$ (Theorems 3-4, Corollary 8). We shall see that the result is a generalization of the case when $K$ is a circle (sphere) centered at the origin and the radius is small enough for the focal sets to be disjoint. Another general observation is about the equidistant function belonging to the pointwise minima of finitely many positive-valued continuous functions (Theorem 2). It is the pointwise minima of the corresponding equidistant functions. 

Examples will also be given to illustrate how complicated an equidistant set can be even in case of choosing relatively simple focal sets. The pathological examples show that additional restrictions are necessary to have computable classes of equidistant sets. In the paper we investigate the following special case: one of the focal sets is the epigraph of a positive-valued, convex and continuously differentiable function and the focal set under the epigraph is a circ\-le (sphere) centered at the origin. We give effective parametric expressions of the equidistant points (Theorems 5-8, Theorems 10-11) and, assuming higher-order differentiability, the corresponding equidistant function will also be characterized (Theorem 9, Theorem 12) as a particular solution of the characterization problem due to M. Ponce and P. Santib\'{a}\~{n}ez \cite{PS}.

\section{General observations}

In what follows, we consider $\mathbb{R}^{n+1}$ as the Cartesian product $\mathbb{R}^{n+1}=\mathbb{R}^n\times \mathbb{R}$. Elements of the form $(x,y)\in \mathbb{R}^n\times \mathbb{R}$ with positive/negative last coordinates form the positive/negative half space. Let 
$$L:=\{(x,y)\ | \ f(x)\leq y\}$$
be the epigraph of a continuous function $f\colon \mathbb{R}^n\to \mathbb{R}$ and let $K\subset \mathbb{R}^{n+1}$ be a nonempty, closed subset such that $K\cap L=\emptyset$.

\begin{lemma}
\label{lem:ex}
If there exists a point $(x_*,y_*)\in K$ such that $y_*< f(x)$ for all $x\in\mathbb{R}^n$, then for any $x\in \mathbb{R}^n$ there exists an element $y\in \mathbb{R}$ such that
	\[d((x,y), K)=d((x,y), L)
\]and
	\[\min\left\{\frac{|x-x_*|^2+y_*^2-u^2}{2(y_*-u)},y_*\right\}\leq y<f(x)
\]holds for any equidistant point $(x,y)$, where
	\[u=\min\left\{f(s)\,\big|\,|x-s|\leq |x-x_*|\right\}.
\]
\end{lemma}

\begin{proof} We are going to use a continuity argument as follows. Let $x\in \mathbb{R}^n$ be an arbitrary point and consider the continuous function
\begin{equation}
\label{difference:01}
d_{x}(y)=d((x,y), K)-d((x,y), L).
\end{equation}
Since $K$ and $L$ are disjoint, closed sets, if $y\geq f(x)$, then $(x,y)\notin K$ and hence
	\[d_x(y)=d((x,y), K)-d((x,y), L)=d((x,y), K)>0.
\]On the other hand, $f$ is a continuous function, thus it attains its minimum on any compact subset of $\mathbb{R}^n$. Taking
	\[u=\min\left\{f(s)\,\big|\,|x-s|\leq |x-x_*|\right\},
\] we have that $y_*<u$. If
	\[y<\min\left\{\frac{|x-x_*|^2+y_*^2-u^2}{2(y_*-u)},y_*\right\},
\]then
	\[2yy_*-2uy>|x-x_*|^2+y_*^2-u^2
\]
and, consequently, 
	\[u^2-2uy>|x-x_*|^2-2yy_*+y_*^2 \ \ \Rightarrow \ \ (u-y)^2>|x-x_*|^2+(y-y_*)^2.
\]Consider the ball
	\[B=\left\{(s,t)\in\mathbb{R}^{n+1}\,\big|\,|s-x|^2+(t-y)^2\leq |x-x_*|^2+(y-y_*)^2\right\}
\]centered at $(x,y)$. If $(s,t)\in B$ and $|s-x|>|x-x_*|$, then $|t-y|<|y-y_*|=y_*-y$ and thus $t-y<y_*-y$, which implies $t<y_*$. This shows that such a point of $B$ cannot be an element of $L$, because $y_*<f(x)$ for all $x\in\mathbb{R}^n$. On the other hand, if $(s,t)\in B$ and $|s-x|\leq |x-x_*|$ then
	\[(t-y)^2\leq |s-x|^2+(t-y)^2\leq |x-x_*|^2+(y-y_*)^2< (u-y)^2
\]as we have seen above, hence $|t-y|<|u-y|$. Here $y<y_*<u$, thus $|t-y|<u-y$, which implies $t<u$. This shows that such a point of $B$ cannot be an element of $L$, because $u\leq f(s)$, when $|x-s|\leq |x-x_*|$.

All together we see that, if
	\[y<\min\left\{\frac{|x-x_*|^2+y_*^2-u^2}{2(y_*-u)},y_*\right\},
\]then the ball $B$ contains the point of $(x_*,y_*)\in K$, but doesn't intersect $L$. Thus $d_x(y)=d((x,y), K)-d((x,y), L)<0$. Since $d_x$ is a continuous function and $d_x(y)>0$ whenever $y\geq f(x)$, there must be a real number $y\in\mathbb{R}$ such that $d_x(y)=0$, which means that $d((x,y), K)=d((x,y), L)$. Furthermore, if a point $(x,y)$ doesn't satisfy
	\[\min\left\{\frac{|x-x_*|^2+y_*^2-u^2}{2(y_*-u)},y_*\right\}\leq y<f(x)
\]then either $d((x,y), K)>d((x,y), L)$, or $d((x,y), K)<d((x,y), L)$, thus it cannot be an equidistant point.
\end{proof}

\begin{theorem} Let $K\subset \mathbb{R}^{n+1}$ be a compact set. The vertical line through $x\in \mathbb{R}^n$ contains an equidistant point for all $x\in \mathbb{R}^n$ if and only if there exists a point $(x_*,y_*)\in K$ such that $y_*< f(x)$ for all $x\in\mathbb{R}^n$.
\end{theorem}

\begin{proof}
If there exists a point $(x_*,y_*)\in K$ such that $y_*< f(x)$ for all $x\in\mathbb{R}^n$, then Lemma \ref{lem:ex} shows that every vertical line intersects the equidistant set $\{K=L\}$. Now assume that for any point $(x,y)\in K$ there exist $x'\in\mathbb{R}^n$ such that $f(x')\leq y$. Since $K$ is compact, there exists a point $(x_*,y_*)\in K$ such that $y_*$ is minimal. Then there exists $x'\in\mathbb{R}^n$ such that $f(x')\leq y_*$. If $y'\geq f(x')$, then $(x',y')\in L$, but $(x',y')\notin K$ because $K$ and $L$ are disjoint. Thus $d\left((x',y'),L\right)=0$ and $d\left((x',y'),K\right)>0$, and hence $(x',y')\notin \{K=L\}$. If $y'<f(x')$, then
	\[y'<f(x')\leq y_*\leq y
\]for any point $(x,y)\in K$. Thus
	\[d\big((x,y),(x',y')\big)=\sqrt{|x-x'|^2+(y-y')^2}\geq \sqrt{|x-x'|^2+(f(x')-y')^2}> |f(x')-y'|
\]
because $x=x'$ would imply that $(x,y)\in K\cap L$. On the other hand, 
	\[|f(x')-y'|=d\big((x',y'),(x',f(x'))\big)\geq d\big((x',y'),L\big).
\]This shows that $d\big((x,y),(x',y')\big)>d\big((x',y'),L\big)$ for any point $(x,y)\in K$. Since $K$ is compact, this implies that $d\big((x',y'),K\big)>d\big((x',y'),L\big)$. Thus $(x',y')\notin \{K=L\}$, and therefore the vertical line through $x'\in\mathbb{R}$ doesn't contain any point of the equidistant set.
\end{proof}

We note that if there exists a point $(x_*,y_*)\in K$ such that $y_*< f(x)$ for all $x\in\mathbb{R}^n$, then the function $f$ is bounded from below and we can choose a new coordinate system such that $f$ is strictly positive. Let's assume, in the remainder of this section, that 
\begin{itemize}
\item $L$ is the epigraph of a continuous, positive-valued function $f\colon \mathbb{R}^n\to \mathbb{R}^+$, 
\item $K\subset \mathbb{R}^{n+1}$ is a nonempty, closed subset such that $K\cap L=\emptyset$ and there exists a point $(x_*,y_*)\in K$ such that $y_*\leq 0$.
\end{itemize}

\begin{definition} The upper/lower equidistant function associated with $K$ and $L$ is defined as
$$G^+(x):=\sup \ \{ y\in \mathbb{R}\ | \ d((x,y),K)=d((x,y),L)\}$$
and
$$G^-(x):=\inf \ \{ y\in \mathbb{R}\ | \ d((x,y),K)=d((x,y),L)\},$$
respectively. The equidistant function associated with $K$ and $L$ is defined by the formula
$$d((x, G(x)),K)=d((x,G(x)),L)$$
provided that the equidistant point is uniquely determined for any $x\in \mathbb{R}^n$. 
\end{definition}

\begin{corollary}
\label{basic:01}
Inequalities $y<G^-(x)$ and $G^+(x)<y$ imply that
$$d((x,y), K)<d((x,y), L)\ \ \textrm{and}\ \ d((x,y), K)>d((x,y), L),$$
respectively.
\end{corollary}

\begin{proof} Consider the continuous function $d_x$ defined by formula \eqref{difference:01}. To avoid zero's (equidistant points) above $G^+(x)$/below $G^-(x)$, the function $d_x$ preserves its sign above $G^+(x)$/below $G^-(x)$. Since $d_x$ takes a positive value at the upper bound $f(x)$ for the last coordinate of the equidistant points at $x\in\mathbb{R}^n$, the sign-preserving property shows that
$$0<d_x(y)=d((x,y),K)-d((x,y),L)$$ 
holds for any $y>G^+(x)$. Let
	\[h(x)=\min\left\{\frac{|x-x_*|^2+y_*^2-u^2}{2(y_*-u)},y_*\right\}
\] be the lower bound for the last coordinate of the equidistant points at $x\in\mathbb{R}^n$. As we have seen in the proof of Lemma \ref{lem:ex}, if $y<h(x)$ then $d_x(y)<0$ and the sign-preserving property shows that
$$0>d_x(y)=d((x,y),K)-d((x,y),L)$$ 
holds for any $y< G^-(x)$. 
\end{proof}

\begin{corollary} 
\label{basic:02}
If the equidistant point is uniquely determined at $x\in \mathbb{R}^n$, then inequalities $y<G(x)$ and $G(x)<y$ imply that
$$d((x,y), K) < d((x,y), L)\ \ \textrm{and}\ \ d((x,y), K) > d((x,y), L),$$
respectively. 
\end{corollary}

\begin{corollary} 
\label{basic:03}
If the equidistant point is uniquely determined at $x\in \mathbb{R}^n$, then inequalities 
$$d((x,y), K) < d((x,y), L)\ \ \textrm{and}\ \ d((x,y), K) > d((x,y), L)$$
imply that $y<G(x)$ and $G(x)<y$, respectively. 
\end{corollary}

\begin{proof}
Suppose that $d((x,y), K) < d((x,y), L)$. We have three possible cases: $G(x)<y$, $G(x)=y$ or $y<G(x)$. At first, by Corollary \ref{basic:01}, it follows that $G^+(x)=G(x)<y$ implies the inequality $d((x,y), K) > d((x,y), L)$, which is a contradiction. $G(x)=y$ is also impossible because $(x,y)$ is not an equidistant point (it is closer to $K$ than to $L$). The only possible choice is $y<G(x)$ as was to be proved. The rest of the statement can be proved in a similar way. 
\end{proof}

\begin{corollary}
The upper/lower equidistant function associated with $K$ and $L$ is upper/lower semi-continuous.
\end{corollary}

\begin{proof} Taking a sequence $\lim_{n\to \infty} x_n=x$, Lemma \ref{lem:ex} implies that both $G^+(x_n)$ and $G^-(x_n)$ are bounded. Since the distance-measuring function is Lipschitz continuous, it follows that any accumulation point of both $G^+(x_n)$ and $G^-(x_n)$
is an equidistant point. Therefore
$$\limsup_{n\to \infty} G^+(x_n)\leq G^+(x)\ \ \textrm{and}\ \ \liminf_{n\to \infty} G^-(x_n)\geq G^-(x)$$
as was to be proved.
\end{proof}

\begin{corollary}
If the equidistant point is uniquely determined at $x\in \mathbb{R}^n$, then the upper/lower equidistant function associated with $K$ and $L$ is continuous at $x\in \mathbb{R}^n$.
\end{corollary}

\begin{corollary}
If the equidistant point is uniquely determined  for any $x\in \mathbb{R}^n$, then the equidistant function associated with $K$ and 
$L$ is continuous.
\end{corollary}

\subsection{Equidistant functions and the minimum operator}

In what follows we are going to investigate the upper/lower equidistant functions belonging to the pointwise minima of finitely many positive-valued, continuous functions. Let $I$ be a finite, nonempty index set and consider the family $\left\{f_i\mid i\in I\right\}$ of positive-valued, continuous functions defined on $\mathbb{R}^n$. The focal set $K\subset \mathbb{R}^{n+1}$ in common is a nonempty, closed subset such that $K$ extends into the negative (closed) half-space and $K\cap L_i=\emptyset$, where 
$$L_i:=\{(x,y)\ | \ f_i(x)\leq y\}$$
is the epigraph of the function $f_i$ ($i\in I$). Let us define the function  
$$f_{\min}\colon \mathbb{R}^n\to\mathbb{R},\quad f_{\min}(x)=\min \left\{f_i(x)\mid i\in I\right\}$$
by taking the pointwise minima. It is easy to see that the epigraph of $f_{\min}$ can be given as $L=\cup_{i\in I\ } L_i$ and $K\cap L=\emptyset$.

\begin{theorem}
\label{basic:01c} If $G_{\min}^+$ and $G_{\min}^-$ are the upper and lower equidistant functions associated with $K$ and $L$, then
$$\min \left\{G_i^-(x)\mid i\in I\right\}\leq G_{\min}^-(x)\leq G_{\min}^+(x)\leq \min \left\{G_i^+(x)\mid i\in I\right\}$$
for any $x\in\mathbb{R}^n$, where $G^+_i$ and $G_i^-$ are the upper and lower equidistant functions associated with $K$ and $L_i$ $(i\in I)$.
\end{theorem}

\begin{proof} 
Let $x\in \mathbb{R}^n$ be an arbitrary point and suppose that 
$$y>\min \left\{G_i^+(x)\mid i\in I\right\}.$$
Then there exists at least one index $i\in I$ such that $y>G_i^+(x)$ and Corollary \ref{basic:01} implies that $d((x,y),L_i)<d((x,y),K)$. Therefore 
$$d((x,y),L)=d((x,y),\cup_{i\in I\ }L_i)\leq d((x,y),L_i)<d((x,y),K).$$
This means that the last coordinate of any equidistant point to the focal sets $K$ and $L$  must be less than or equal to the minimum value among $G_i^+(x)$ ($i\in I$). On the other hand, suppose that 
$$y<\min \left\{G_i^-(x)\mid i\in I\right\}.$$
Then $y<G_i^-(x)$ for all $i\in I$ and Corollary \ref{basic:01} implies that $d((x,y),K)<d((x,y), L_i)$ for all $i\in I$. Thus  
$$d((x,y),K)<\min \left\{d((x,y), L_i)\mid i\in I\right\}=d((x,y),\cup_{i\in I\ } L_i)=d((x,y), L)$$
and we have that the last coordinate of any equidistant point to the focal sets $K$ and $L$  must be greater than or equal to the minimum value among $G_i^-(x)$ ($i\in I$).
\end{proof}

\begin{corollary} If $G_i^-(x)=G_i^+(x)$ for any $i\in I$, then $G_{\min}^+(x)=G_{\min}^-(x)$. If the equidistant sets belonging to the focal sets $K$ and $L_i$ are equidistant functions for any $i\in I$, then so is the equidistant set belonging to the focal sets $K$ and $L$. 
\end{corollary}

\subsection{A necessary and sufficient condition for the existence of equidistant functions}

\begin{theorem}
Let $L$ be the epigraph of a continuous, positive-valued function $f\colon \mathbb{R}^n\to\mathbb{R}$ and let $K\subset \mathbb{R}^{n+1}$ be a closed set and suppose that $K\cap L=\emptyset$. If there exist no point $x_0\in\mathbb{R}^n$ and real numbers $0<y_1<y_2$, such that $(x_0,y_1)\notin K$ and $(x_0,y_2)\in K$, then the vertical line contains at most one equidistant point for any $x\in \mathbb{R}^n$.
\end{theorem}

\begin{proof}
It's enough to prove that if there exists a point $x_1\in\mathbb{R}^n$ and real numbers $y_3<y_4$ such that both $T_1=(x_1,y_3)$ and $T_2=(x_1,y_4)$ are elements of the equidistant set $\{K=L\}$, then there exist $x_0\in\mathbb{R}^n$ and real numbers $0<y_1<y_2$ such that $P_1=(x_0,y_1)\notin K$ and $P_2=(x_0,y_2)\in K$. For the sake of clarity, we are going to use the partial ordering of the elements in $\mathbb{R}^{n+1}$ with respect to the last coordinate.

\begin{figure}
\includegraphics[height=6.5cm]{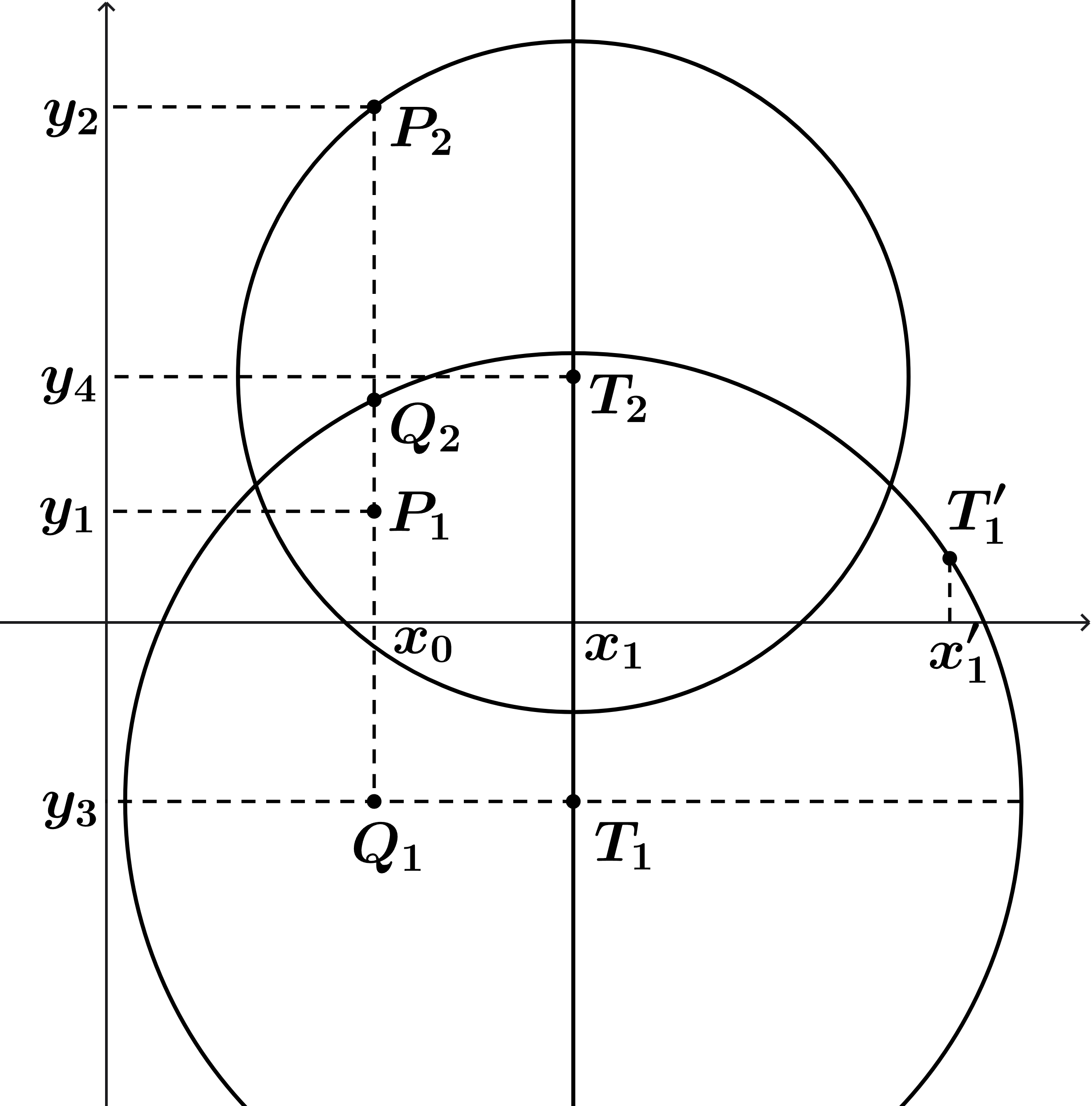}
\caption{\label{forwardfigure}}
\end{figure}

Let $r_1$ denote the common distance of $T_1$ from $K$ and $L$, and let $r_2$ denote the common distance of $T_2$ from $K$ and $L$ (see Figure 1). Since $K\cap L=\emptyset$, we have $r_1>0$ and $r_2>0$. Consider the ball $S_1$ centered at $T_1$ with radius $r_1$ and the ball $S_2$ centered at $T_2$ with radius $r_2$. The epigraph of $f$ is closed, therefore it must intersect the boundary of the ball $S_1$ in at least one point $T_1'=(x_1',y_1')$, but the epigraph can't intersect the interior of either $S_1$ or $S_2$. In particular, $T_1'\geq T_1$, that is $T_1'$ is on the upper hemisphere of $S_1$. Furthermore, $y_1'>0$, because $f$ is a positive function. In a similar way, the set $K$ must intersect the boundary of the ball $S_2$ in at least one point $P_2=(x_0,y_2)$, but $K$ can't intersect the interior of either $S_1$ or $S_2$. Therefore
$$d^2(T_1', T_1)=\left|x_1'-x_1\right|^2+\left(y_1'-y_3\right)^2=r_1^2$$
because $T_1'$ is on the boundary of $S_1$, but
$$d^2(T_1',T_2)=\left|x_1'-x_1\right|^2+\left(y_1'-y_4\right)^2\geq r_2^2$$
because $T_1'$ is not in the interior of $S_2$. Thus we have that
$$r_2^2-r_1^2\leq -2y_1'(y_4-y_3)+y_4^2-y_3^2 \ \ \Rightarrow\ \ y_1' \leq \frac{r_1^2-r_2^2+y_4^2-y_3^2 }{2(y_4-y_3)}$$
because $T_1<T_2$. In a similar way,   
$$d^2(P_2, T_2)=\left|x_0-x_1\right|^2+\left(y_2-y_4\right)^2=r_2^2,$$
$$d^2(P_2,T_1)=\left|x_0-x_1\right|^2+\left(y_2-y_3\right)^2\geq r_1^2$$
and we have that
\begin{equation}
\label{thm03:eq1}
r_1^2-r_2^2\leq -2y_2(y_3-y_4)+y_3^2-y_4^2 \ \ \Rightarrow\ \ y_2 \geq \frac{r_1^2-r_2^2+y_4^2-y_3^2 }{2(y_4-y_3)}\geq y_1'>0,
\end{equation}
that is $P_2=(x_0,y_2)\in K$, where $y_2>0$. To construct a point $P_1=(x_0,y_1)\notin K$ such that $0<y_1<y_2$ we claim that $r_2<r_1$ and 
\begin{equation}
\label{thm03:eq2}
\left|x_0-x_1\right|^2\leq \left|x_1'-x_1\right|^2.
\end{equation}
If $T_2\leq T_1'$ then inequality \eqref{thm03:eq2} is automatic because $T_1'$ is not in the interior of $S_2$, but $P_2$ is on its boundary:
$$\left|x_1'-x_1\right|^2+(y_1'-y_4)^2\geq \left|x_0-x_1\right|^2+(y_2-y_4)^2=r_2^2.$$
Using \eqref{thm03:eq1}, we have that $T_2 \leq T_1'\leq P_2$ and, consequently,
$$(y_1'-y_4)^2\leq (y_2-y_4)^2\ \ \Rightarrow\ \ \left|x_1'-x_1\right|^2\geq \left|x_0-x_1\right|^2.$$
In particular, $T_1 < T_2 \leq T_1'$ implies that $r_2\leq d(T_1', T_2)<d(T_1',T_1)=r_1$ because $T_1'$ is not in the interior of $S_2$ and $T_1<T_2$ along the same vertical line. Discussing the opposite case, suppose that $T_1'< T_2$. Since $S_2$ and the interior of the epigraph of the function $f$ must be disjoint, the vertical ray through $T_1'$ (strictly) above $y_1'$ can't intersect $S_2$, that is 
$$\left|x_0-x_1\right|^2 \leq \left|x_0-x_1\right|^2 +(y_2-y_4)^2=d^2(P_2, T_2)=r_2^2 < \left|x_1'-x_1\right|^2 \leq r_1^2$$
and the proof of inequality \eqref{thm03:eq2} together with inequality $r_2<r_1$ is complete. Finally, let us consider the upper hemisphere of $S_1$. It is the graph of the function
	\[g\colon D_1\to\mathbb{R},\quad g(x)=y_3+\sqrt{r_1^2-\left|x_1-x\right|^2},
\]where $D_1\subset\mathbb{R}^n$ is the $n$-dimensional ball centered at $x_1$ with radius $r_1$. The function $g$ is continuous and $g(x_1')=y_1'>0$. In the sense of inequality \eqref{thm03:eq2},
	\[g(x_0)=y_3+\sqrt{r_1^2-\left|x_1-x_0\right|^2}\geq y_3+\sqrt{r_1^2-\left|x_1-x_1'\right|^2}=g(x_1')=y_1'>0.
\]This shows that $g(x_0)\geq y_1'\geq y_3$, however $g(x_0)=y_3$ is not possible, because it would imply that $r_2\geq r_1$, but that's false. Then any point of the segment connecting $Q_1=(x_0,y_3)$ and $Q_2=(x_0,g(x_0))$ is an interior point of $S_1$ including the endpoint $Q_1$, but excluding $Q_2$. If $y_3>0$, let's choose $y_1=y_3>0$, otherwise let $y_1=g(x_0)/2 >0$. Then $P_1=(x_0,y_1)$ is clearly an interior point of $S_1$ such that $y_1>0$. The set $K$ can't intersect the interior of $S_1$, thus $P_1\notin K$ and $y_2\geq g(x_0)>y_1$ because $P_2$ is not in the interior of $S_1$. To sum up, there exist $x_0\in\mathbb{R}^n$ and real numbers $0<y_1<y_2$ such that $P_1=(x_0,y_1)\notin K$ and $P_2=(x_0,y_2)\in K$.
\end{proof}

\begin{theorem}
Let $K\subset \mathbb{R}^{n+1}$ be a convex compact set. If there exist $x_1\in\mathbb{R}^n$ and real numbers $0<y_1<y_2$ such that $(x_1,y_1)\notin K$ and $(x_1,y_2)\in K$, then there exist a positive-valued, convex and, consequently, continuous function $f\colon \mathbb{R}^n\to \mathbb{R}$ and a point $x_2\in\mathbb{R}^n$ such that the epigraph $L$ of $f$ is disjoint to $K$, and the vertical line through $x_2$ intersects the equidistant set $\{K=L\}$ in at least two distinct points.
\end{theorem}

\begin{proof}
Let $P_0=(x_*,y_*)\in K$ be a point of $K$, where $y_*$ is minimal. It's clear that $y_*\leq y_2$ and the point $P_1=(x_1,y_1)\notin K$ can't be between the points $P_0\in K$ and $P_2=(x_1,y_2)\in K$, because $K$ is convex. Thus $0<y_1< y_*$ or $x_*\neq x_1$. If $x_*\neq x_1$, then let $w=x_1-x_*\in\mathbb{R}^n$, otherwise let $w\in\mathbb{R}^n$ be an arbitrary nonzero vector. Let $H$ denote the 2-dimensional subspace spanned by the vectors $(w,0)$ and $e_{n+1}$, where $e_{n+1}$ is the last element of the standard basis of $\mathbb{R}^{n+1}$. Let $Q_2=(x_*+s\cdot w,t)$ be a point of $K$ in the plane $x_*+H$, where $s$ is maximal. If $y_1< y_*$, then the point $Q_1=(x_*+s\cdot w,y_1)$ does not belong to $K$ and $t>y_1$ because $y_*\leq y$ for any $(x,y)\in K$. If $y_*\leq y_1$, then $x_*\neq x_1$ and $P_1\notin K$ is a point in the triangle $P_0P_2Q_1$, where $P_0\in K$, $P_2\in K$, thus $Q_1$ can't be an element of $K$. Furthermore, $t>y_1$, because otherwise $P_1\notin K$ was a point in the triangle $P_0P_2Q_2$, which would contradict to the convexity of $K$. Thus, in all possible cases, $Q_2\in K$, $Q_1\notin K$ and $0<y_1<t$ (see Figure \ref{conversefigure} (A)).

The choice of the point $Q_2$ implies that no point of the horizontal, open ray $R=\left\{Q_2+(h\cdot w,0)\,|\,h> 0 \right\}$ is an element of $K$. Moreover, no vertical line through a point of the open ray $R$ can intersect $K$. Let $H_1$ denote the hyperplane orthogonal to $e_{n+1}$ through $Q_2$ and let $K'$ denote the orthogonal projection of $K$ to $H_1$. Then no point of $K$ is projected onto the open ray $R$. Thus the points of $R$ are exterior points of the projection $K'$ and hence $Q_2$ is a boundary point of $K'$. The orthogonal projection $K'$ is also a convex set, thus there exist a hyperplane $H_2$ through $Q_2$ with normal vector $(v,0)\in\mathbb{R}^n$ such that every point of $K'$ is contained in the same closed half space bounded by $H_2$. Then $H_2$ is a supporting hyperplane of $K$ too at $Q_2$, and it contains all points of the vertical line through $Q_2$. We can assume that the normal vector $(v,0)$ points outward.

\begin{figure}
\def\tabularxcolumn#1{m{#1}}
\begin{tabularx}{\linewidth}{@{}cXX@{}}

\begin{tabular}{cc}
\subfloat[]{\includegraphics[height=6cm]{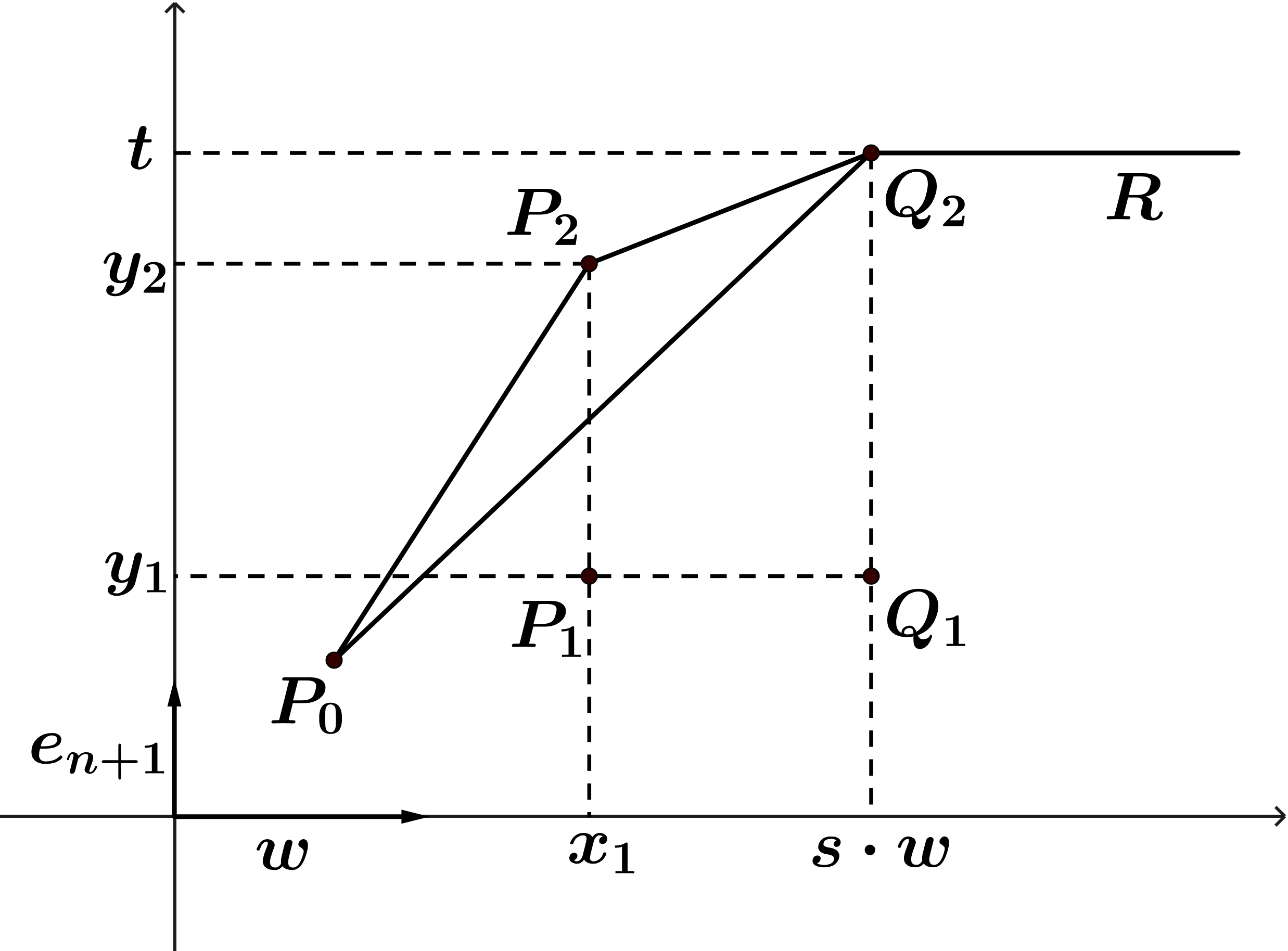}} 
   & \subfloat[]{\includegraphics[height=7cm]{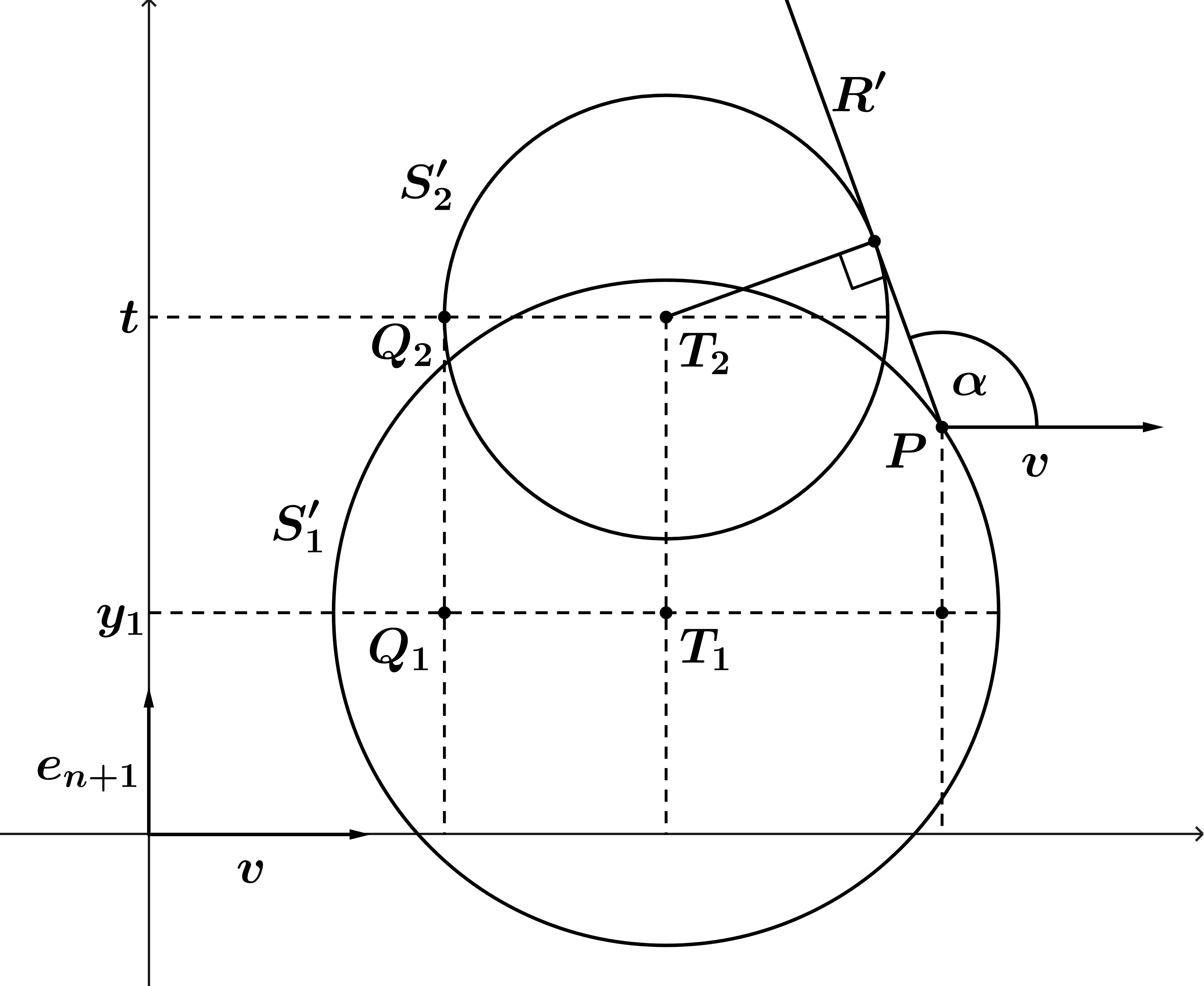}}
\end{tabular}

\end{tabularx}

\caption{\label{conversefigure}}
\end{figure}

Now let $T_1=Q_1+(v,0)$ and $T_2=Q_2+(v,0)$. The distance of $T_2$ from $H_2$ is $r_2=\left|v\right|>0$ and the sphere $S_2$ centered at $T_2$ with radius equal to $r_2$ intersects $H_2$ at the point $Q_2$. This shows that the distance of $T_2$ from $K$ is also equal to $r_2$. Since $T_1$ and $T_2$ are two points on the same vertical line, the distance of $T_1$ from $H_2$ is also $r_2$. Let $r_1$ denote the distance of $T_1$ from $K$ and let $S_1$ be the sphere centered at $T_1$ with radius equal to $r_1$. Then $S_1$ contains exactly one point of $K$ on its boundary. It is clear that $r_1<r_2$ is not possible, because otherwise $S_1$ wouldn't intersect $H_2$. Also, $r_1=r_2$ is not possible, because in such a case $S_1$ would intersect $H_2$ only in the point $Q_1$ which is not an element of $K$. Therefore $r_2<r_1$ and we note that $Q_2$ can't be inside the sphere $S_1$, thus no point of the upper hemisphere of $S_2$ is inside the sphere $S_1$.

Let $H_3$ be the 2-dimensional plane through $Q_1$ parallel to the subspace spanned by the vectors $(v,0)$ and $e_{n+1}$. The plane $H_3$ contains all the points $Q_1$, $Q_2$, $T_1$ and $T_2$. The intersection of $H_3$ with the spheres $S_1$ and $S_2$ are circles $S_1'$ and $S_2'$ centered at $T_1$ and $T_2$ with radii $r_1$ and $r_2$, respectively (see Figure 2 (B)). Given $0<\varepsilon <r_1/r_2-1$, let $P$ be the upper intersection point of $S_1'$ with the vertical line through the point $T_1+(1+\varepsilon)\cdot (v,0)$, and consider the ray $R'$ in the plane $H_3$ whose endpoint is $P$ and which intersects the circle $S_2'$ in exactly one point. Such a point $P$ and the ray $R'$ exist because $1<1+\varepsilon<r_1/r_2$ and no point of the upper hemisphere of $S_2$ is inside the sphere $S_1$. It's clear that if $\varepsilon$ tends to $0$, then the tangent of the angle $\alpha$ enclosed by the ray $R'$ and the horizontal vector $(v,0)$ tends to $-\infty$. Thus if $l_y$ denotes the length of the orthogonal projection of $K$ to the vertical axis, then it is possible to choose $\varepsilon >0$ small enough to make $|\tan \alpha|$ larger than $l_y/r_2$ and that the ray $R'$ has no further intersection point with $S_1$ besides $P$.

Now we define the function $f\colon \mathbb{R}^n\to\mathbb{R}$, $f(x)=y_P+|x-x_P|\cdot|\tan\alpha|$, where $y_P>y_1>0$ denotes the last coordinate of $P=(x_P,y_P)$. Then $f$ is a convex function with positive values and the graph of $f$ is above the convex set $K$, because $|x-x_P|> r_2$ for any $(x,y)\in K$ and the last coordinate of any point of $K$ is less than $l_y$. Furthermore, the graph of $f$ intersects each of the spheres $S_1$ and $S_2$, but no point of the graph is inside any of these spheres. Thus the distance of $T_1$ from the epigraph of $f$ is $r_1$ and the distance of $T_2$ from the epigraph of $f$ is $r_2$. This shows that both the points $T_1=(x_*+s\cdot w+v,y_1)$ and $T_2=(x_*+s\cdot w+v,t)$ are elements of the equidistant set $\{K=L\}$, where $L$ is the epigraph of $f$. In other words, the vertical line through $x_2=x_*+s\cdot w+v$ intersects the equidistant set $\{K=L\}$ in at least two distinct points $T_1$ and $T_2$.
\end{proof}

\begin{corollary}
\label{sum:01}
Let $K\subset \mathbb{R}^{n+1}$ be a convex compact set which has a point $(x_*,y_*)\in K$ such that $y_*\leq 0$. Then $G^-(x)=G^+(x)$ for all $x\in\mathbb{R}^n$ with the epigraph $L$ of \underline{any} positive-valued, continuous/convex function $f\colon \mathbb{R}^n\to\mathbb{R}$, such that $K\cap L=\emptyset$, if and only if there exist no point $x_0\in\mathbb{R}^n$ and real numbers $0<y_1<y_2$, such that $(x_0,y_1)\notin K$ and $(x_0,y_2)\in K$.
\end{corollary}

\subsection{Examples: focal sets with more than one equidistant points along the vertical lines}

Consider the points $P_1=\left(x_0,-\frac{1}{2}\right)$, $P_2=\left(x_0,\frac{1}{2}\right)$, $Q_1=(0,y_0)$ and $Q_2=(2x_0,y_0)$ in $\mathbb{R}^2$, where
	\[x_0=\frac{\sqrt{3}}{\sqrt{3}-1},\quad y_0=\frac{1}{2}\cdot \frac{\sqrt{3}+1}{\sqrt{3}-1}.
\]Let $K$ be the disc centered at $Q_1$ with radius $R=2$, and let $S$ be the circle centered at $Q_2$ with the same radius $R=2$. Here $x_0>R$, thus the vertical line through $x_0$ doesn't intersect any of the disc $K$ or the circle $S$, and it separates them into two different half planes. Furthermore, $\frac{1}{2}<y_0<R$, hence $K$ has a point, whose second coordinate is negative (see Figure \ref{cantorfigure}). Now we define the function $g\colon \left[-\frac{1}{2},\frac{1}{2}\right]\to \mathbb{R}^2$ such that $g(t)$ is the intersection point of $S$ with the line segment connecting $Q_2$ to the point $(x_0,t)$. Then
	\[g(t)=\left(2x_0-R\cos\left(\arctan\left(\frac{y_0-t}{x_0}\right)\right),y_0-R\sin\left(\arctan\left(\frac{y_0-t}{x_0}\right)\right)\right).
\]Since
	\[\frac{y_0-\left(-\frac{1}{2}\right)}{x_0}=\frac{\frac{1}{2}\cdot \frac{\sqrt{3}+1}{\sqrt{3}-1}+\frac{1}{2}}{\frac{\sqrt{3}}{\sqrt{3}-1}}=\frac{\frac{1}{2}\cdot (\sqrt{3}+1)+\frac{1}{2}(\sqrt{3}-1)}{\sqrt{3}}=1
\]and
	\[\frac{y_0-\frac{1}{2}}{x_0}=\frac{\frac{1}{2}\cdot \frac{\sqrt{3}+1}{\sqrt{3}-1}-\frac{1}{2}}{\frac{\sqrt{3}}{\sqrt{3}-1}}=\frac{\frac{1}{2}\cdot (\sqrt{3}+1)-\frac{1}{2}(\sqrt{3}-1)}{\sqrt{3}}=\frac{1}{\sqrt{3}},
\]we see that
	\[P_1'=g\left(-\frac{1}{2}\right)=\left(2x_0-R\cos\left(\frac{\pi}{4}\right),y_0-R\sin\left(\frac{\pi}{4}\right)\right)=\left(2x_0-\sqrt{2},y_0-\sqrt{2}\right)
\]and
	\[P_2'=g\left(\frac{1}{2}\right)=\left(2x_0-R\cos\left(\frac{\pi}{6}\right),y_0-R\sin\left(\frac{\pi}{6}\right)\right)=\left(2x_0-\sqrt{3},y_0-1\right).
\]Then $P_1'$ and $P_2'$ are above the $x$-axis. The function $g$ is continuous and the second coordinate of $g(t)$ is strictly increasing, thus $g$ is a one-to-one correspondence between the elements of $\left[-\frac{1}{2},\frac{1}{2}\right]$ and the points of the shorter circular arc of $S$ between $P_1'$ and $P_2'$, which we denote by $S'$. Let $l$ be the line through $P_2'$ perpendicular to the line segment $\overline{P_2'Q_2}$. The equation of $l$ is
	\[\sqrt{3}\,x+y-2\sqrt{3}\,x_0-y_0+4=0.
\]It can be easily seen that $l$ separates $Q_1$ and $Q_2$. The distance of $l$ from $Q_1$ is
	\[d(Q_1,l)=\frac{\left|\sqrt{3}\cdot 0+y_0-2\sqrt{3}\,x_0-y_0+4\right|}{\sqrt{3+1}}=\frac{\left|4-2\sqrt{3}\,x_0\right|}{2}=\left|2-\sqrt{3}\,x_0\right|>2=R
\]Thus the line $l$ is disjoint to the disc $K$ and $K$ is below the line $l$. Now we define the continuous function
	\[f\colon \mathbb{R}\to\mathbb{R},\quad f(x)=\begin{cases}
	-\sqrt{3}\cdot x+2\sqrt{3}\,x_0+y_0-4,& \textrm{if } x\leq 2x_0-\sqrt{3},\\
	y_0-\sqrt{R^2-(x-2x_0)^2},& \textrm{if } 2x_0-\sqrt{3}<x<2x_0-\sqrt{2},\\
	y_0-\sqrt{2},& \textrm{if } 2x_0-\sqrt{2}\leq x.
	\end{cases}
\]The graph of $f$ is the union of the circular arc $S'$ (connecting $P_1'$ and $P_2'$), and two rays: one of them is the horizontal ray on the right of $P_1'$, the other one is the ray consisting of the points of $l$ on the left of $P_2'$. Hence $f$ is a continuous, positive-valued function, whose epigraph $L$ is disjoint to $K$. Moreover, $f$ is a convex function and $L$ is a convex set.

It's clear, that if $-\frac{1}{2}\leq t\leq \frac{1}{2}$, then the distance of the point $(x_0,t)$ measured from $L$ is the same as the distance measured from $K$. Thus
	\[d((x_0,t),L)=d((x_0,t),S)=d((x_0,t),Q_2)-R=\sqrt{|x_0|^2+|y_0-t|^2}-R
\]
	\[=d((x_0,t),Q_1)-R=d((x_0,t),K).
\]This shows that every point of the vertical line segment $\overline{P_1P_2}$ is an element of the equidistant set $\{K=L\}$.

\begin{figure}
\def\tabularxcolumn#1{m{#1}}
\begin{tabularx}{\linewidth}{@{}cXX@{}}

\begin{tabular}{cc}
\subfloat[]{\includegraphics[height=5.6cm]{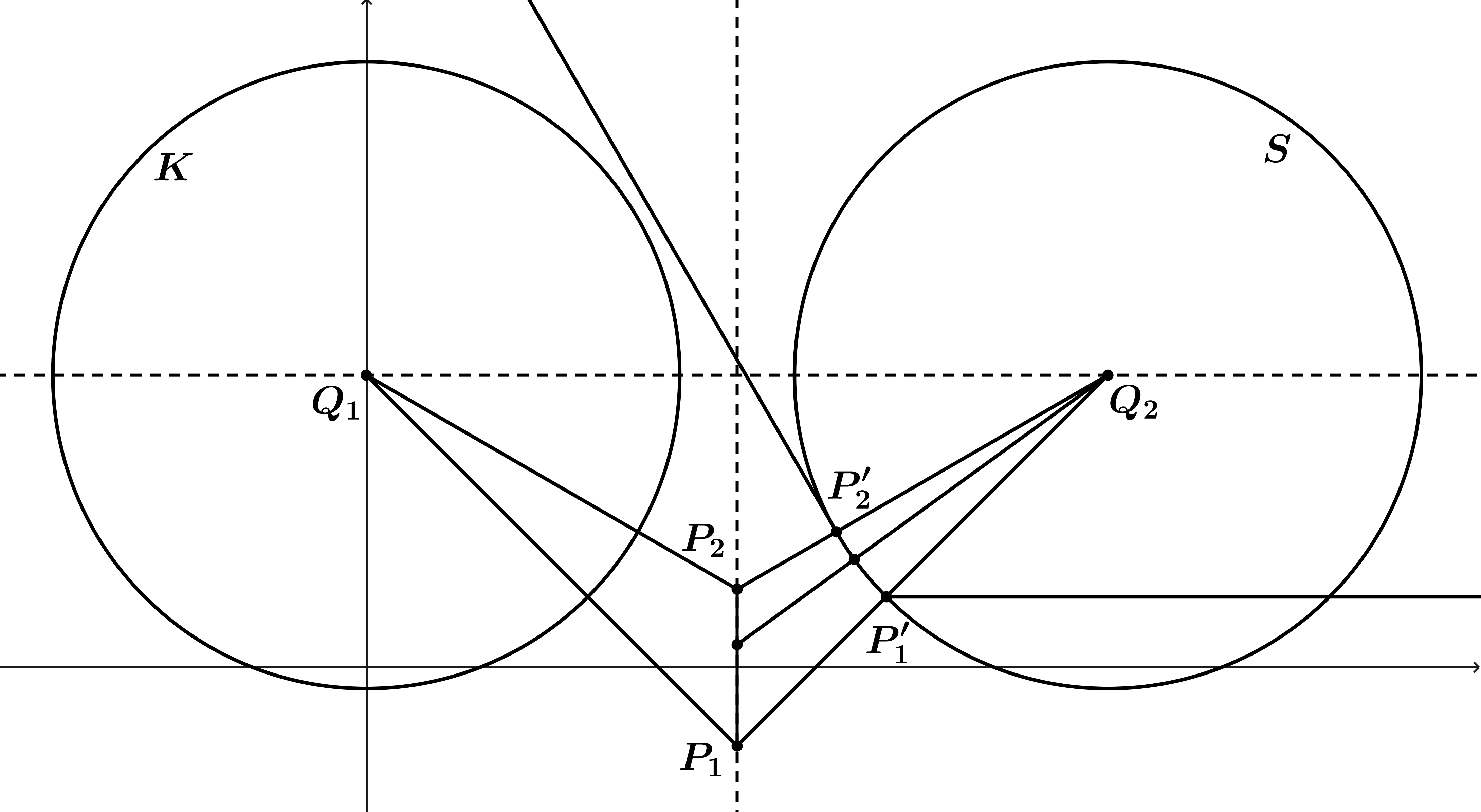}} 
   & \subfloat[]{\includegraphics[height=7cm]{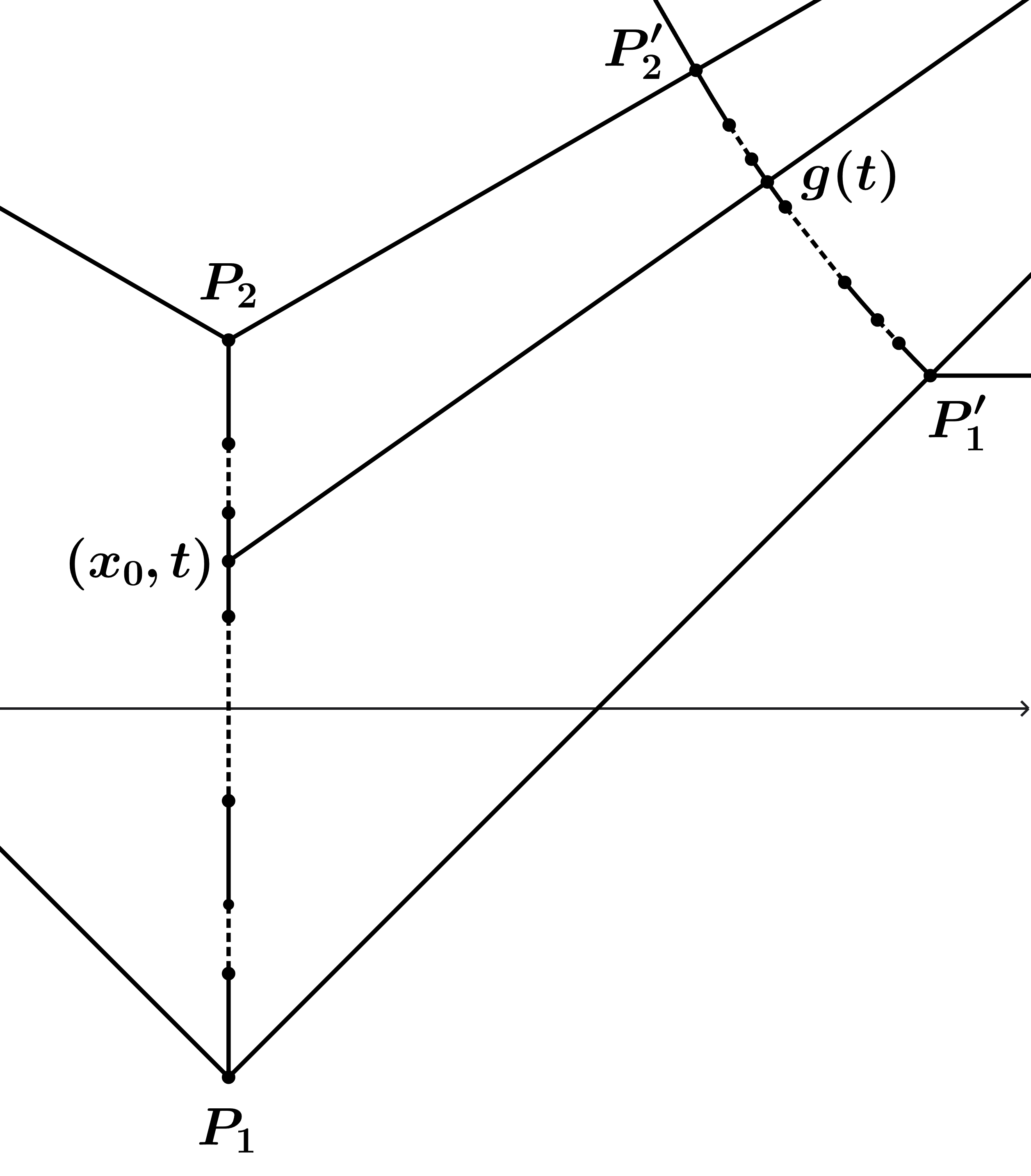}}
\end{tabular}

\end{tabularx}

\caption{\label{cantorfigure}}
\end{figure}

The Smith-Volterra-Cantor set (also known as the fat Cantor set) is constructed by subtracting a countable union of pairwise disjoint open intervals from the $[0,1]$ interval. The Smith-Volterra-Cantor set contains no interval, yet it has a positive Lebesgue measure. Now we will modify the set $L$ above such that the intersection of the equidistant set $\{K=L\}$ with the vertical line through $x_0$ is a shifted version of the Smith-Volterra-Cantor set.

Let $]a_n,b_n[$ be the $n$-th interval subtracted from the $[0,1]$ interval during the construction of the Smith-Volterra-Cantor set. Then consider the points
	\[T_n=g\left(a_n-\frac{1}{2}\right)\quad\textrm{and}\quad U_n=g\left(b_n-\frac{1}{2}\right)
\]and let $S_n$ denote the disk segment bounded by the chord $\overline{T_nU_n}$ of the disk centered at $Q_2$ with radius $R=2$. Since the function $g$ is one-to-one and the intervals $]a_n,b_n[$ are pairwise disjoint, the disk segments $S_n$ are also pairwise disjoint. Then let's remove all the disk segments $S_n$ together with their interior to get the set $L'$, but we don't remove the points of the chords $\overline{T_nU_n}$. Hence $L'$ is a closed set and it is convex as the intersection of the convex set $L$ with half-planes bounded by the lines $\overleftrightarrow{T_nU_n}$.

It's clear that, for any $t\in[-\frac{1}{2},\frac{1}{2}]$ the distance of the point $(x_0,t)$ from the disk $K$ is the same as its distance from the circle $S$, which is the same as the distance from the point $g(t)$ of the circular arc $S'$. If there exists $n\in \mathbb{N}$ such that $t\in \big]a_n-\frac{1}{2},b_n-\frac{1}{2}\big[$, then the point $g(t)$ is not an element of $L'$, thus the distance of $(x_0,t)$ from $K$ is less than its distance from $L'$ and hence $(x_0,t)\notin \{K=L\}$. On the other hand, if $t\notin \big]a_n-\frac{1}{2},b_n-\frac{1}{2}\big[$ for all $n\in \mathbb{N}$, then the point $g(t)$ is in the set $L'$, thus the distance of $(x_0,t)$ from $K$ is the same as its distance from $L'$ and hence $(x_0,t)\in \{K=L\}$. Therefore, the point $(x_0,t)$ is an element of the equidistant set $\{K=L\}$ if and only if $t+\frac{1}{2}$ is not an element of any set $]a_n,b_n[$ removed from the $[0,1]$ interval during the construction of the Smith-Volterra-Cantor set. This clearly shows that the set
	\[\left\{t+\frac{1}{2}\in \mathbb{R}\,\bigg|\,t\in \left[-\frac{1}{2},\frac{1}{2}\right],\,(x_0,t)\in \{K=L\}\right\}
\]is the Smith-Volterra-Cantor set. Hence the intersection of $\{K=L\}$ with the vertical line through $x_0$ contains no interval, yet it has a positive Lebesgue measure.

\section{Circles and spheres}

Let $L:=\{(x,y)\ | \ f(x)\leq y\}$ be the epigraph of a continuous, positive-valued function $f\colon \mathbb{R}^n\to \mathbb{R}^+$ and $K\subset \mathbb{R}^{n+1}$ be a solid sphere of dimension $n+1$. It is centered at the origin with radius $R>0$ such that $K\cap L=\emptyset$, that is
$$x^2+f^2(x) > R^2 \quad (x\in \mathbb{R}),\ \ \textrm{or} \ \ |x|^2+f^2(x) > R^2 \quad (x\in \mathbb{R}^n)$$
depending on the dimension of the embedding space $\mathbb{R}^{n+1}$ ($n=1$ or $n\geq 2$). Since the focal sets are assumed to be disjoint, there are certainly no equidistant points inside $K$. Therefore it can be considered a sphere/ball or a circle/circular disk in an equivalent way.  

\begin{lemma} For any $x\in \mathbb{R}^n$, the equidistant point $(x,y)$ belonging to the focal sets $K$ and $L$ is uniquely determined.
\end{lemma}

\begin{proof} Although the statement is a direct consequence of Corollary \ref{sum:01}, we present a more simple direct proof as well.  
Suppose, in the contrary, that $(x_0, y_1)$ and $(x_0, y_2)$ are equidistant points such that $y_1 < y_2$. Let $(x_1, f(x_1))$ and $(x_2, f(x_2))$ be the points of the epigraph, where the distances 
$$d((x_0, y_1),L)=d((x_0, y_1),(x_1, f(x_1)))\ \ \textrm{and}\ \ d((x_0, y_2),L)=d((x_0, y_2),(x_2, f(x_2)))$$
are attained at. Introducing the polar distances
$$r_1=\sqrt{|x_0|^2+y_1^2} \ \ \textrm{and}\ \ r_2=\sqrt{|x_0|^2+y_2^2}$$
from the origin, we can write that
$$|x_1-x_0|^2+(f(x_1)-y_1)^2=(r_1-R)^2$$
because $(x_0,y_1)$ is an equidistant point. On the other hand, 
$$|x_1-x_0|^2+(f(x_1)-y_2)^2\geq (r_2-R)^2$$
because 
$$d((x_0,y_2), (x_1,f(x_1)))\geq d((x_0, y_2),L)=d((x_0, y_2),(x_2, f(x_2)))=r_2-R.$$
Therefore
$$(f(x_1)-y_2)^2-(f(x_1)-y_1)^2\geq (r_2-R)^2-(r_1-R)^2$$
and, consequently,
$$f(x_1)(y_1-y_2)\geq R(r_1-r_2).$$
Since $f$ is positive-valued and $y_1<y_2$, the left-hand side is negative. This means that $r_1<r_2$, that is $r_1-R<r_2-R$.
Using that $(x_1, f(x_1))$ is on the sphere centered at $(x_0, y_1)$ with radius $r_1-R$, we have that the vertical ray 
$$\{ (x_1, y)\ | \ f(x_1)\leq y\} \subset L$$
intersects the interior of the sphere centered at a higher vertical position $(x_0, y_2)$ with a larger radius $r_2-R$, that is
$$r_2-R=d((x_0,y_2),(x_2, f(x_2)))>d((x_0, y_2), L)$$
which is a contradiction because $(x_0, y_2)$ is an equidistant point. 
\end{proof}

\begin{definition} The equidistant function associated with $K$ and $L$ is called a circular/spherical equidistant function depending on the dimension of the embedding space $\mathbb{R}^{n+1}$. 
\end{definition}

\subsection{The parametric expressions of the equidistant points in the plane}

Let $K\subset \mathbb{R}^2$ be a circle centered at the origin of the Euclidean plane and consider the focal set $L$ as the epigraph of a positive-valued, continuously differentiable convex function $f\colon \mathbb{R}\to \mathbb{R}^+$ such that $K\cap L=\emptyset$, that is
$$t^2+f^2(t) > R^2 \quad (t\in \mathbb{R}),$$
where $R$ is the radius of the circle. The outer unit normal to the graph is
$$N=\frac{1}{\sqrt{1+f'^{2}}} (f', -1).$$

\begin{theorem} The equidistant points can be written in the parametric form
\begin{equation} 
\label{eqcirc:x}
 x(t)=t + \frac{f'(t)}{2\sqrt{1+f'^{2}(t)}}\cdot\frac{t^2+f^2(t)-R^2}{R-\alpha(t)}
\end{equation}
and
\begin{equation} 
\label{eqcirc:y}
y(t)=f(t) - \frac{1}{2\sqrt{1+f'^{2}(t)}}\cdot\frac{t^2+f^2(t)-R^2}{R-\alpha(t)},
\end{equation}
where the parameter $t\in \mathbb{R}$ satisfies inequality
\begin{equation}
\label{condcirc:01} 
\alpha(t):=\frac{tf'(t)-f(t)}{\sqrt{1+f'^2(t)}}< R.
\end{equation}
\end{theorem}

\begin{proof}
Let $(x,y)$ be an equidistant point. Using the convexity of the epigraph of the function $f$, consider the uniquely determined closest point of the form $(t,f(t))$ to $(x,y)$. This means that the difference vector $(x,y)-(t,f(t))$ is proportional to the outer unit normal, that is 
$$(x,y)=(t,f(t))+sN(t).$$
In a more detailed form
$$x(t)= t + \frac{s(t) f'(t)}{\sqrt{1+f'^{2}(t)}}\ \ \textrm{and}\ \ y(t) = f(t) - \frac{s(t)}{\sqrt{1+f'^{2}(t)}},$$ 
where 
\begin{equation}
\label{dist:01}
s(t)=\sqrt{x^2(t)+y^2(t)}-R
\end{equation}
is the common distance from the focal sets $L$ and $K$.
Therefore
$$(s(t)+R)^2=x^2(t)+y^2(t)=\left(t + \frac{s(t) f'(t)}{\sqrt{1+f'^{2}(t)}}\right)^2+\left(f(t) - \frac{s(t)}{\sqrt{1+f'^{2}(t)}}\right)^2$$
and, consequently,
\begin{equation}
\label{dist:02}
s(t)=\frac{1}{2}\frac{t^2+f^2(t)-R^2}{R-\alpha(t)}
\end{equation}
provided that inequality \eqref{condcirc:01} is satisfied, where $\alpha$ measures the signed distance of the tangent line to $f$ from the origin\footnote{The signed distance is negative if the origin is under the tangent line of the function $f$ and positive if the origin is over the tangent line of $f$.}. 
\end{proof}

\subsubsection{Critical parameters} Consider the domain
$$ D:=\{t \in \mathbb{R} \mid \alpha(t)<R \}$$
for the parameters of the expressions \eqref{eqcirc:x} and \eqref{eqcirc:y}; we will call the elements of $D$ \emph{admissible parameters} and the elements of the complement are the \emph{critical parameters}. The origin is obviously an admissible parameter together with the elements of an open neighbourhood because of
$$\alpha(0)=\dfrac{-f(0)}{\sqrt{1+f'^2(0)}}< R,$$
where $\alpha$ is continuous. Let us consider the set $C^+:=\{t\in \mathbb{R}^+ \mid \alpha(t)\geq R\}$ of positive critical parameters.  

\begin{lemma} If $t_0\in C^+$, then $\alpha(t)\geq \alpha(t_0)$ for any $t>t_0$; that is, $t\in C^{+}$.
\end{lemma}

\begin{proof} The proof is based on the convexity of the function $f$ provided that it is continuously differentiable. Since $f$ is a positive-valued function, $t_0\in C^+$ implies by \eqref{condcirc:01} that $f'(t_0)>0$. Observe that   
$$h_0(x):=\frac{t_0 x-f(t_0)}{\sqrt{1+x^2}}\quad (x>-t_0/f(t_0))$$
is a monotone increasing function because of
$$h_0'(x)=\frac{t_0+f(t_0)x}{(1+x^2)^{3/2}} > 0 \quad (x>-t_0/f(t_0)).$$
Since $0<f'(t_0)\leq f'(t)$, it follows that
$$h_{0}(f'(t))\geq h_{0}(f'(t_0)) \ \  \Rightarrow \ \ \frac{t_0f'(t)-f(t_0)}{\sqrt{1+f'^2(t)}}\geq \frac{t_0f'(t_0)-f(t_0)}{\sqrt{1+f'^2(t_0)}}=\alpha(t_0).$$
On the other hand, by the mean value theorem of differential calculus, 
$$tf'(t)-f(t)-(t_0f'(t)-f(t_0))=(t-t_0)(f'(t)-f'(\xi))\geq 0$$
because $f'$ is a monotone increasing function and $t_0 < \xi < t$. We have that
$$ \alpha(t)=\frac{tf'(t)-f(t)}{\sqrt{1+f'^2(t)}}\geq \frac{t_0f'(t)-f(t_0)}{\sqrt{1+f'^2(t)}}\geq \frac{t_0f'(t_0)-f(t_0)}{\sqrt{1+f'^2(t_0)}}=\alpha(t_0)$$
as was to be proved. 
\end{proof}

\begin{corollary}
$C^+$ is either empty or $C^+=\mathopen]t_0^+, \infty\mathclose[$ where $0 < t_0^+:=\inf C^+$.
\end{corollary}
In a similar way, consider the set $C^-:=\{t\in \mathbb{R}^-\ | \ \alpha(t)\geq R\}$ of the negative critical parameters. $C^-$ is either empty or $C^-=\mathopen]-\infty, t_0^-\mathclose[$, where $0 > t_0^-:=\sup C^-$. Taking the complement of $C^+\cup C^-$, the domain of the  parameters in \eqref{eqcirc:x} and \eqref{eqcirc:y}, that is the set of the admissible parameters is $D=\mathopen]t_0^-, t_0^+\mathclose[$ including the cases $t_0^-=-\infty$ and (or) $t_0^+=\infty$ if there are no negative and (or) positive critical parameters (see Figure 4).

\begin{definition}
\label{equid:par}
The pair of the parametric expressions $x\colon \mathopen]t_0^-, t_0^+\mathclose[\to \mathbb{R}$ and $y\colon \mathopen]t_0^-, t_0^+\mathclose[\to \mathbb{R}$ is called the equidistant parameterization for the graph of the equidistant function.
\end{definition}

\begin{figure}[h!]
\begin{subfigure}{.3\linewidth}
\includegraphics[width=.99\linewidth]{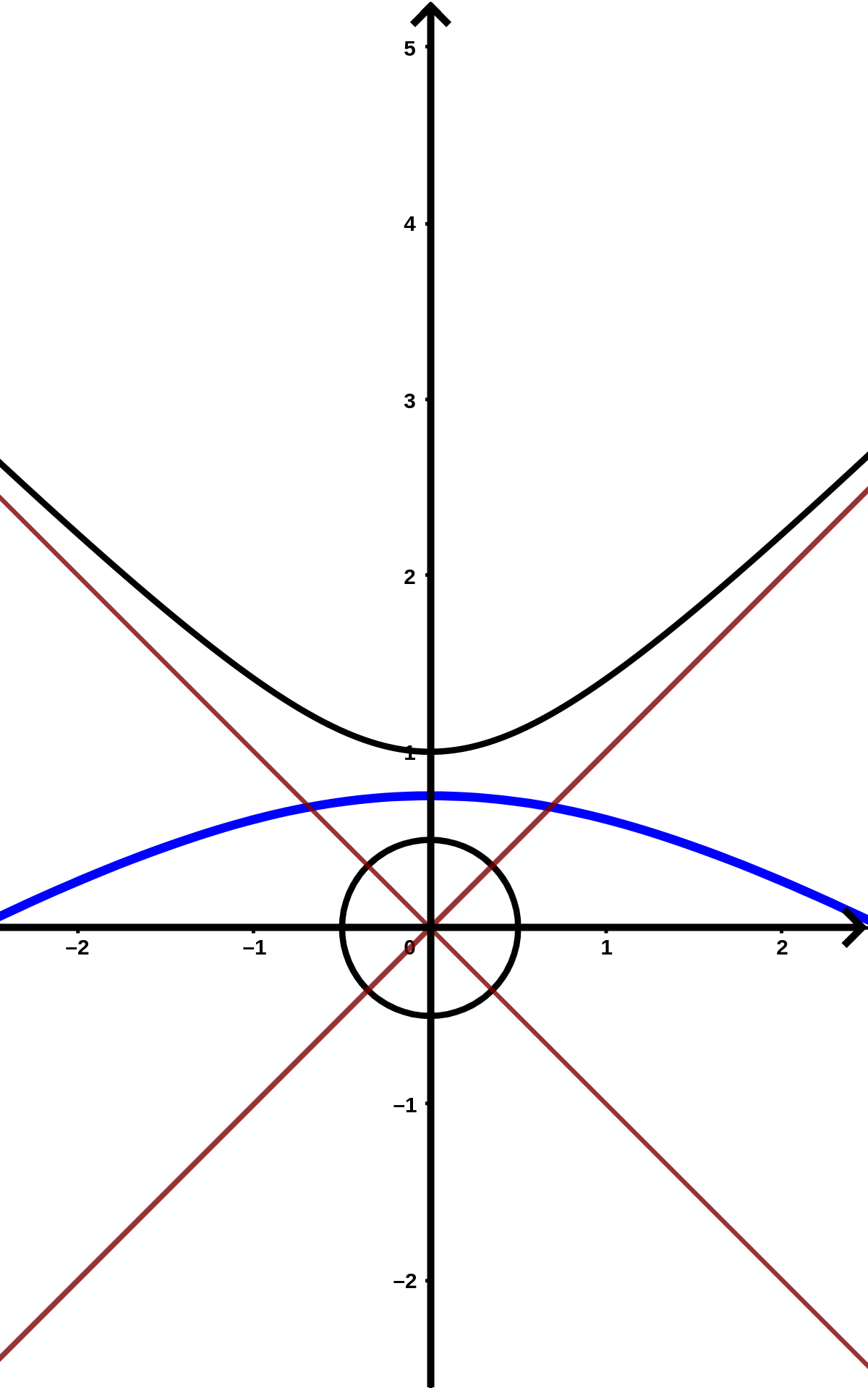}
\caption{ $f(t)=\sqrt{t^2+1}$ }
\end{subfigure} \hfill 
\begin{subfigure}{.3\linewidth}
\includegraphics[width=.99\linewidth]{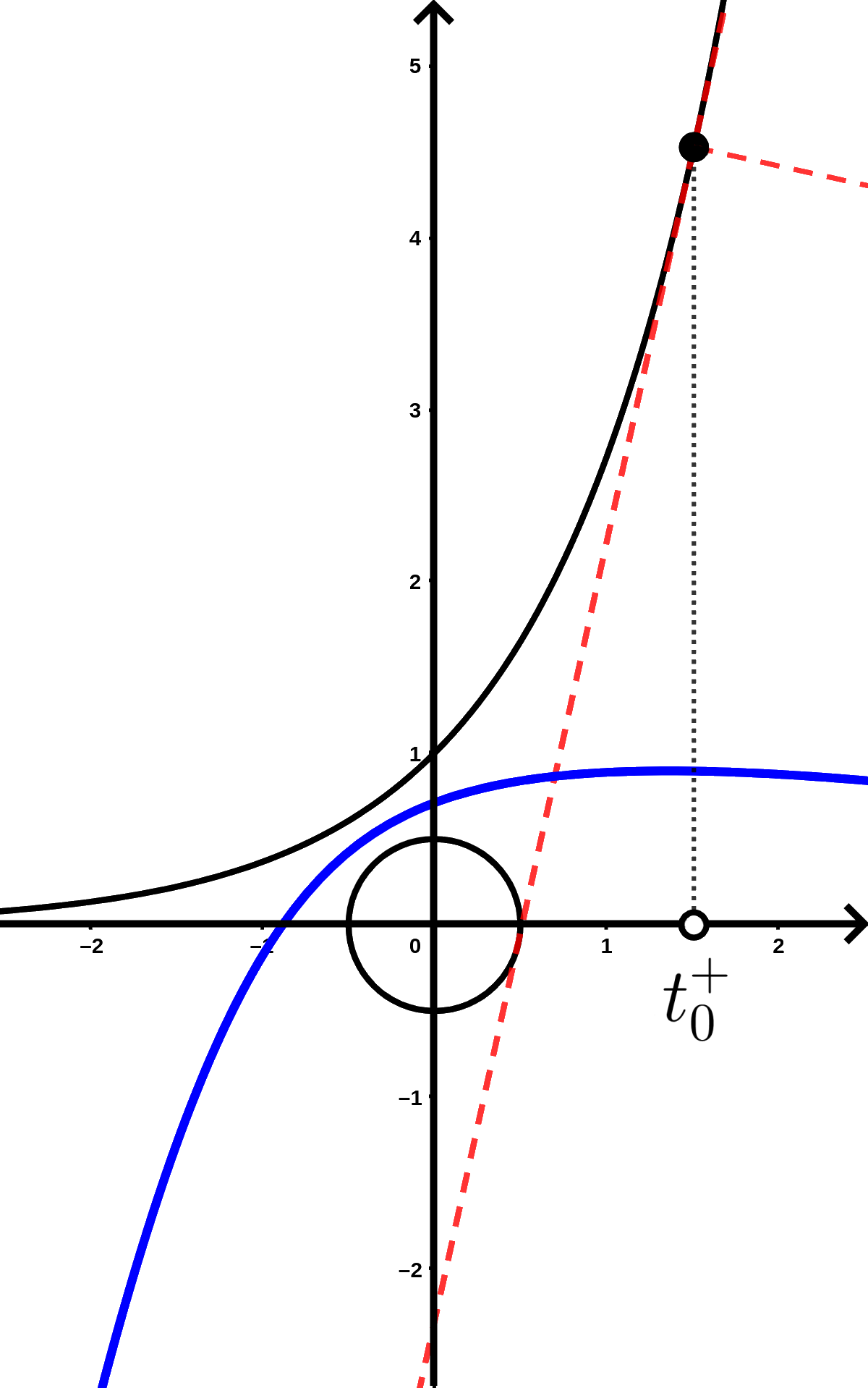}
\caption{ $f(t)=e^t$ }
\end{subfigure} \hfill
\begin{subfigure}{.3\linewidth}
\includegraphics[width=.99\linewidth]{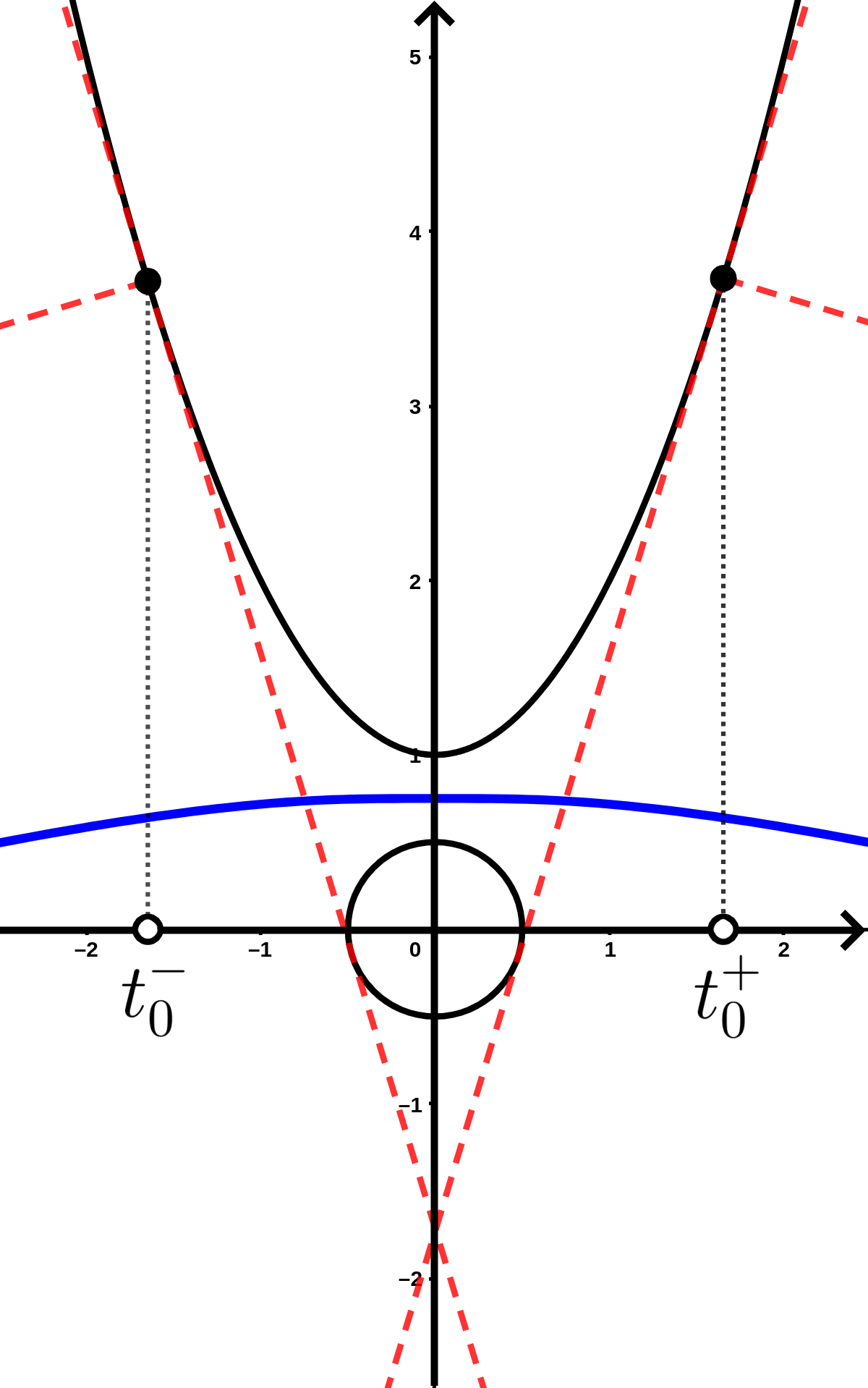}
\caption{ $f(t)=t^2+1$ }
\end{subfigure}
\caption{Some examples demonstrating the cases when (A) there are no critical parameters, so the parameter domain is $D=\mathopen]-\infty, \infty\mathclose[$; \ (B) there are only positive critical parameters: $D=\mathopen]-\infty, t_0^+\mathclose[$;  \ (C) there are both negative and positive critical parameters: $D=\mathopen]t_0^-, t_0^+\mathclose[$.}
\end{figure}

\subsubsection{The expression of the functions $f$ and $f'$ in terms of the equidistant parameterization \eqref{eqcirc:x} and \eqref{eqcirc:y}} According to formulas \eqref{dist:01} and \eqref{dist:02}, we have that
\begin{equation}
\label{square}
\frac{t^2+f^2(t)-R^2}{R-\alpha(t)}=\frac{\sqrt{x^2(t)+y^2(t)}-R}{1/2},
\end{equation}
where 
\begin{equation}
\label{condcirc:11}
x^2(t)+y^2(t)>R^2.
\end{equation}
Using equation \eqref{square}, the parametric expressions can be written as
\begin{equation} 
\label{eqcirc:x1}
 x(t)=t + \frac{\sqrt{x^2(t)+y^2(t)}-R}{\sqrt{1+f'^{2}(t)}}f'(t),
\end{equation}
and
\begin{equation} 
\label{eqcirc:y1}
y(t)=f(t) - \frac{\sqrt{x^2(t)+y^2(t)}-R}{\sqrt{1+f'^{2}(t)}}.
\end{equation}
Equation \eqref{eqcirc:x1} implies that
\begin{equation}
\label{Condcirc:01}
\frac{x(t)-t}{\sqrt{x^2(t)+y^2(t)}-R}=\frac{f'(t)}{\sqrt{1+f'^{2}(t)}},
\end{equation}
that is
\begin{equation}
\label{condcirc:02}
 -1 < \frac{x(t)-t}{\sqrt{x^2(t)+y^2(t)}-R} < 1
\end{equation}
and
$$\frac{1}{\sqrt{1+f'^{2}(t)}}=\sqrt{1-\left (\frac{f'(t)}{\sqrt{1+f'^{2}(t)}}\right)^2}=\sqrt{1-\left (\frac{x(t)-t}{\sqrt{x^2(t)+y^2(t)}-R}\right )^2}=$$
$$\frac{\sqrt{\left (\sqrt{x^2(t)+y^2(t)}-R\right)^2-(x(t)-t)^2}}{\sqrt{x^2(t)+y^2(t)}-R}.$$
Equation \eqref{eqcirc:y1} shows that
\begin{equation}
\label{circfunction}
f(t)=y(t)+\sqrt{\left (\sqrt{x^2(t)+y^2(t)}-R\right)^2-(x(t)-t)^2}.
\end{equation}
Using \eqref{eqcirc:x1} and \eqref{eqcirc:y1}, we can also express the derivative function in terms of the equidistant parameterization as follows:
\begin{equation}
\label{derivfunction:circ}
f'(t)=\frac{x(t)-t}{\sqrt{\left(\sqrt{x^2(t)+y^2(t)}-R\right)^2-(x(t)-t)^2}}.
\end{equation}

\subsubsection{The case of a twice continuously differentiable function: the compatibility of the equidistant parameterization}

Let $f$ be a twice continuously differentiable (positive-valued, convex) function. In this case the equidistant parameterization consists of continuously differentiable functions. The compatibility condition for the equidistant parameterization means the compa\-ti\-bility of the right-hand sides of formulas \eqref{circfunction} and \eqref{derivfunction:circ}, that is
$$\left(y(t)+\sqrt{\left (\sqrt{x^2(t)+y^2(t)}-R\right)^2-(x(t)-t)^2}\right)'=\frac{x(t)-t}{\sqrt{\left(\sqrt{x^2(t)+y^2(t)}-R\right)^2-(x(t)-t)^2}}.$$
\begin{theorem} The compatibility condition for the equidistant parameterization is equivalent to 
\begin{equation}
\label{compcirc}
\left(g(t)-\frac{x(t)}{r(t)}\right)x'(t)=\left(\sqrt{1-g^2(t)}+\frac{y(t)}{r(t)}\right)y'(t),
\end{equation}
where
$$r(t):=\sqrt{x^2(t)+y^2(t)} \ \ \textrm{and} \ \ g(t):=\frac{x(t)-t}{\sqrt{x^2(t)+y^2(t)}-R}=\frac{x(t)-t}{r(t)-R}.$$
\end{theorem}

\begin{proof}
In terms of the auxiliary functions, 
\begin{equation}
\label{circcond:0}
f(t)=y(t)+(r(t)-R)\sqrt{1-g^2(t)},
\end{equation}
\begin{equation}
\label{circcond:00}
g(t)=\frac{f'(t)}{\sqrt{1+f'^2(t)}}\ \ \Rightarrow \ \ f'(t)=\frac{g(t)}{\sqrt{1-g^2(t)}}.
\end{equation}
Therefore
\begin{equation}
\label{circcond:000}
f'(t)=y'(t)+r'(t)\sqrt{1-g^2(t)}-\left(r(t)-R\right)\frac{g(t)g'(t)}{\sqrt{1-g^2(t)}},
\end{equation}
$$g'(t)=\frac{(x'(t)-1)(r(t)-R)-(x(t)-t)r'(t)}{(r(t)-R)^2}=\frac{x'(t)-1-g(t)r'(t)}{r(t)-R},$$
$$r'(t)= \frac{x(t)x'(t)+y(t)y'(t)}{r(t)}. $$
Thus, using \eqref{circcond:00} and \eqref{circcond:000}, the compatibility equation is equivalent to
\begin{align*}
\frac{g(t)}{\sqrt{1-g^2(t)}} &= y'(t)+\sqrt{1-g^2(t)} \, r'(t)-\frac{g(t)}{\sqrt{1-g^2(t)}} \left (x'(t)-1-g(t)r'(t) \right),  \\
g(t) &= \sqrt{1-g^2(t)} \, y'(t)+(1-g^2(t)) \, r'(t)-g(t)\left(x'(t)-1-g(t)r'(t) \right), \\
0 &= \sqrt{1-g^2(t)} \, y'(t) -g(t)x'(t) +r'(t), \\
0 &= \sqrt{1-g^2(t)} \, y'(t) -g(t)x'(t) + \frac{x(t)x'(t)+y(t)y'(t)}{r(t)}, \\
0 &= \left(\sqrt{1-g^2(t)}+\frac{y(t)}{r(t)}\right)y'(t)-\left(g(t)-\frac{x(t)}{r(t)}\right)x'(t)
\end{align*}
as was to be proved.
\end{proof}

\begin{remark} \label{compcoeffpos} {\emph{In \eqref{compcirc}, the coefficient of $y'(t)$ on the right-hand side is always strictly positive. This is because supposing the converse leads to
\begin{align*}
\sqrt{1-g^2(t)}+\dfrac{y(t)}{r(t)} &\leq 0, \\
y(t) &\leq -r(t) \, \sqrt{1-g^2(t)}, \\
f(t)-(r(t)-R)\sqrt{1-g^2(t)} &\leq -r(t) \, \sqrt{1-g^2(t)}, \\
f(t)+R\sqrt{1-g^2(t)} &\leq 0,
\end{align*}
which is obviously a contradiction. This also allows us to rearrange \eqref{compcirc} for $y'(t)$.}}
\end{remark}

\begin{theorem}
\label{derivnonzero:x}
The function $x\colon \mathopen]t_0^-, t_0^+\mathclose[\to \mathbb{R}$,
$$ x(t)=t + \frac{f'(t)}{2\sqrt{1+f'^{2}(t)}}\cdot\frac{t^2+f^2(t)-R^2}{R-\alpha(t)}$$
is a one-to-one correspondence with nowhere vanishing derivative.
\end{theorem}

\begin{proof} Since for any $x\in \mathbb{R}$ we have a uniquely determined equidistant point $(x,y)$ together with a uniquely determined closest point $(t,f(t))$, where the distance of $(x,y)$ to the epigraph is attained at, the function $x\colon \mathopen]t_0^-, t_0^+\mathclose[\to \mathbb{R}$ is a one-to-one correspondence. Using that
$$\alpha'(t)=\frac{f''(t)(t+f(t)f'(t))}{(1+f'^2(t))^{3/2}},$$
a straightforward calculation shows that
$$x'(t)=1+\frac{1}{2}\frac{f''(t)(t^2+f^2(t)-R^2)(R+f(t)\sqrt{1+f'^2(t)})}{(1+f'^2(t))^{3/2}(R-\alpha(t))^2}+\frac{f'(t)(t+f(t)f'(t))}{\sqrt{1+f'^2(t)}(R-\alpha(t))},$$
where
$$\frac{f'(t)(t+f(t)f'(t))}{\sqrt{1+f'^2(t)}(R-\alpha(t))}=\frac{\alpha(t)+f(t)\sqrt{1+f'^2(t)}}{R-\alpha(t)}=-1+\frac{R+f(t)\sqrt{1+f'^2(t)}}{R-\alpha(t)},$$
that is
$$x'(t)=\frac{1}{2}\frac{f''(t)(t^2+f^2(t)-R^2)(R+f(t)\sqrt{1+f'^2(t)})}{(1+f'^2(t))^{3/2}(R-\alpha(t))^2}+\frac{R+f(t)\sqrt{1+f'^2(t)}}{R-\alpha(t)}>0$$
as was to be proved.
\end{proof}

\begin{remark} {\emph{Another, less direct proof can be given as follows: suppose that $x'(t)=0$. The compatibility condition \eqref{compcirc} shows that $y'(t)=0$. Using \eqref{eqcirc:x} and \eqref{eqcirc:y}, we can write the parametric expressions into the form
$$x(t)=t+\varphi(t)f'(t) \ \ \textrm{and} \ \ y(t)=f(t)-\varphi(t).$$
Therefore
$$0=1+\varphi'(t)f'(t)+\varphi(t)f''(t)\ \ \textrm{and}\ \ 0=f'(t)-\varphi'(t).$$
Finally,
$$0=1+f'^2(t)+\varphi(t)f''(t)\geq 1,$$
which is a contradiction. The indirect method will be especially useful in case of higher dimensional spaces.}}
\end{remark}

\subsubsection{Necessary and sufficient conditions for a pair $x(t)$ and $y(t)$ to be the equidistant parameterization of the graph of the  equidistant function belonging to $K$ and $L$}

Let $f$ be a twice continuously differentiable (positive-valued, convex) function. Condition
\begin{equation}
\label{necandsuf:01}
r(t):=\sqrt{x^2(t)+y^2(t)} > R
\end{equation}
is obvious because $r(t)-R$ is the common distance of the equidistant point $(x(t),y(t))$ from the focal sets $K$ and $L$. Introducing the auxiliary function
$$g(t):=\frac{x(t)-t}{\sqrt{x^2(t)+y^2(t)}-R},$$
equation \eqref{circcond:00} shows that $|g(t)|<1$ and
$$\frac{g(t)}{\sqrt{1-g^2(t)}}$$
must be monotone increasing to provide the convexity of the function $f$. In the sense of \eqref{circcond:0}, the expression
$$f(t):=y(t)+\left(r(t)-R\right)\sqrt{1-g^2(t)}$$
must be positive. It is obvious in case of $y(t)\geq 0$. If $y(t)<0$ then the positivity is equivalent to
$$1-g^2(t) > \frac{y^2(t)}{(r(t)-R)^2} \ \ \Rightarrow \ \ 1>g^2(t)+\frac{y^2(t)}{\left (r(t)-R\right)^2}.$$
Combining the cases $y(t)\geq 0$ and $y(t)<0$ into a single formula, we can write that
$$1>g^2(t)+\frac{1-\sgn y(t)}{2} \cdot \frac{y^2(t)}{\left (r(t)-R\right)^2}.$$

\begin{theorem}
\label{equidpar2d:thm}
Let $x\colon \mathopen]t_0^-, t_0^+\mathclose[\to \mathbb{R}$ and $y\colon \mathopen]t_0^-, t_0^+\mathclose[\to \mathbb{R}$ be continuously differentiable functions defined on a possibly degenerate, open interval containing the origin, and suppose that 
\begin{itemize}
\item [(i)] the polar distance is greater than the radius of the circle $K$ centered at the origin, that is
$$r(t):=\sqrt{x^2(t)+y^2(t)} > R,$$
\item[(ii)] the absolute value of the auxiliary function
$$g(t):=\frac{x(t)-t}{r(t)-R}$$
is less than one,
$$1>g^2(t)+\frac{1-\sgn y(t)}{2} \cdot\frac{y^2(t)}{\left (r(t)-R\right)^2},$$
\item [(iii)] the function $$\frac{g(t)}{\sqrt{1-g^2(t)}}$$ is monotone increasing and the compatibility condition \eqref{compcirc} is satisfied.
\end{itemize}
Then the functions $x$ and $y$ give the equidistant parameterization of the graph of the equidistant function belonging to $f\colon \mathbb{R}\to \mathbb{R}$, where
\begin{equation}
\label{function2d:thm}
f(t)=y(t)+\left(r(t)-R\right)\sqrt{1-g^{2}(t)}
\end{equation}
for any $t\in \mathopen]t_0^-, t_0^+\mathclose[$. If $x\colon \mathopen]t_0^-, t_0^+\mathclose[\to \mathbb{R}$ is a one-to-one correspondence with nowhere vanishing derivative, then its domain is maximal for the equidistant parameterization. 
\end{theorem}

\begin{proof} As we have seen above, the conditions allow the application of the formula \eqref{function2d:thm} to give a well-defined, positive-valued function. Using the compatibility condition, its derivative can be written as
$$f'(t)=\frac{g(t)}{\sqrt{1-g^2(t)}}$$
and, consequently, it is a convex, twice continuously differentiable function due to (iii). (Observe that the derivative of the function contains only the functions $x$ and $y$ without their derivatives.) Let us consider the difference vector pointing from $(t,f(t))$ to $(x(t),y(t))$. Its coordinates at the parameter $t$ are
$$d_1(t)=x(t)-t= (r(t)-R) g(t),$$
$$d_2(t)=y(t)-f(t)=-(r(t)-R)\sqrt{1-g^2(t)}.$$
It is easy to see that the difference vector is normal to the graph of $f$ at the parameter $t$ because
$$ \left( d_1(t), d_2 (t)\right) \begin{pmatrix} 1 \\ f'(t) \end{pmatrix} = (r(t)-R) \left( g(t)-\sqrt{1-g^2(t)}\frac{g(t)}{\sqrt{1-g^2(t)}} \right)=0. $$
In particular, it is a downward-pointing normal since the second coordinate is negative. Thus, the distance of $(x(t),y(t))$ from the graph of $f$ is attained at $(t,f(t))$:
$$\sqrt{d_1^2(t)+d_2^2(t)}=(r(t)-R) \sqrt{g^2(t)+(1-g^2(t))}= r(t)-R=d((x(t),y(t)),K). $$
This means that for any $t\in \mathopen]t_0^-, t_0^+\mathclose[$, the equidistant point closest to the graph point $(t,f(t))$ is $(x(t), y(t))$, that is we have the equidistant parametrization of the equidistant function.
\end{proof}

\begin{exercise} Prove that the equidistant function is the envelope of the upper branches of the one-parameter family of hyperbolas one of whose focal points runs along the function $f$ and the other one is uniformly given at the origin.  
\end{exercise}
Hint. Using that we have an equidistant parameterization, 
$$\sqrt{(x(t)-t)^2+(y(t)-f(t))^2}=r(t)-R$$
that is the equidistant point at the parameter $t$ satisfies equation
$$\sqrt{(x-t)^2+(y-f(t))^2}-\sqrt{x^2+y^2}=-R$$
of the hyperbola with focal points $(t,f(t))$ and $(0,0)$ such that it is closer to $(t,f(t))$ than to the origin. On the other hand, the gradient of the hyperbola at the parameter $t$ is
$$\left(\frac{x(t)-t}{\sqrt{(x(t)-t)^2+(y(t)-f(t))^2}}-\frac{x(t)}{r(t)}, \frac{y(t)-f(t)}{\sqrt{(x(t)-t)^2+(y(t)-f(t))^2}}-\frac{y(t)}{r(t)}\right),$$ 
where
$$\frac{x(t)-t}{\sqrt{(x(t)-t)^2+(y(t)-f(t))^2}}-\frac{x(t)}{r(t)}=\frac{x(t)-t}{r(t)-R}-\frac{x(t)}{r(t)}=g(t)-\frac{x(t)}{r(t)},$$
$$\frac{y(t)-f(t)}{{\sqrt{(x(t)-t)^2+(y(t)-f(t))^2}}}-\frac{y(t)}{r(t)}=\frac{y(t)-f(t)}{r(t)-R}-\frac{y(t)}{r(t)}=-\left(\sqrt{1-g^2(t)}+
\frac{y(t)}{r(t)}\right).$$ 
The compatibility condition \eqref{compcirc} shows that the gradient of the hyperbola is orthogonal to the tangent vector of the equidistant function. Therefore they have the same tangent line point by point.

\subsubsection{The characterization of equidistant functions} 
Let $f$ be a twice continuously differentiable, positive-valued, convex function. In what follows we are looking for expressions of the equidistant parameterization in terms of the equidistant function $G=y\circ x^{-1}$. Since $G(x(t))=y(t)$, it follows that 
$$G'(x(t))x'(t)=y'(t)$$
and the compatibility condition \eqref{compcirc} shows that
$$G'(x(t))=\frac{g(t)-x_0(t)}{\sqrt{1-g^2(t)}+y_0(t)},$$
where
$$x_0(t)=\frac{x(t)}{r(t)}, \ \ y_0(t)=\frac{y(t)}{r(t)}\ \ \textrm{and}\ \ \sqrt{1-g^2(t)}+y_0(t)>0$$
because of $f(t)>0$. Therefore 
$$\sqrt{1-g^2(t)}G'(x(t))=g(t)-\left(x_0(t)+y_0(t)G'(x(t))\right)$$
and, by taking the square of both sides, we have a quadratic equation
$$0=(1+G'^2(x))g^2-2(x_0+y_0 G'(x))g+(x_0+y_0 G'(x))^2-G'^2(x).$$
Using that $x_0^2+y_0^2=1$, a straightforward calculation shows that
$$g_1(t)=\frac{x_0(t)(1-G'^2(x(t)))+2y_0(t) G'(x(t))}{1+G'^2(x(t))} \ \ \textrm{or} \ \ g_2(t)=x_0(t).$$
It is easy to see that $g(t)=x_0(t)$ implies that $G'(x(t))=0$ and, consequently, $g_1(t)=g_2(t)$. Thus we have\footnote{The solution \eqref{equidfunction:01} can also be written into the form
$$g(t)=x_0(t)+2\frac{y_0(t)-x_0(t)G'(x(t))}{1+G'^2(x(t))}G'(x(t)).$$
This is the formal analogue of the formula for higher dimensional spaces.} that 
\begin{equation}
\label{equidfunction:01}
g(t)=\frac{x_0(t)(1-G'^2(x(t)))+2y_0(t) G'(x(t))}{1+G'^2(x(t))}.
\end{equation}
In a more detailed form,                                                                
$$\frac{x(t)-t}{r(t)-R}=\frac{x_0(t)(1-G'^2(x(t)))+2y_0(t)G'(x(t))}{1+G'^2(x(t))}$$
and
$$t=x(t)-\frac{r(t)-R}{r(t)}\cdot\frac{x(t)(1-G'^2(x(t)))+2y(t)G'(x(t))}{1+G'^2(x(t))},$$
where
$$y(t)=G(x(t))\ \ \textrm{and} \ \ r(t)=\sqrt{x^2(t)+G^2(x(t))}.$$
Introducing the function
\begin{equation}
\label{equidfunction:02}
h(x):=x-\frac{\sqrt{x^2+G^2(x)}-R}{\sqrt{x^2+G^2(x)}}\cdot\frac{x(1-G'^2(x))+2G(x)G'(x)}{1+G'^2(x)},
\end{equation}
it follows that $h(x(t))=t$. Since the function \eqref{equidfunction:02}
depends only on the equidistant function  we can formulate the following result as an application of Theorem \ref{equidpar2d:thm}.

\begin{theorem}
Let $G\colon \mathbb{R}\to \mathbb{R}$ be a twice continuously differentiable function and suppose that the function $h\colon \mathbb{R}\to \mathopen]t_0^-, t_0^+\mathclose[$ defined by formula \eqref{equidfunction:02} is a one-to-one correspondence with nowhere vanishing derivative and its range is a possibly degenerate, open interval containing the origin. The function $G$ is an equidistant function if and only if the pair of functions $x(t):=h^{-1}(t)$ and $y(t):=G(x(t))$ is the equidistant parameterization for the graph of $G$. 
\end{theorem}

\begin{remark} {\emph{As we have seen above, the focal sets are the circle centered at the origin with radius $R$ such that $\sqrt{x^2+G^2(x)}>R$ and the epigraph of the positive-valued, twice continuously differentiable, convex function $f\colon \mathbb{R}\to \mathbb{R}$ defined by formula \eqref{function2d:thm}.}}
\end{remark} 

\subsubsection{Angular relationships} 

$$\begin{tblr}{|c|c|c|c|}
\hline 
& \sin & \cos & \tan \\ 
\hline
\rho_f(t) & \dfrac{f(t)}{\sqrt{t^2+f^2(t)}} &
\dfrac{t}{\sqrt{t^2+f^2(t)}} &
\dfrac{f(t)}{t} \\ 
\hline 
\rho_e(t) & \dfrac{y(t)}{r(t)} &
\dfrac{x(t)}{r(t)} &
\dfrac{y(t)}{x(t)} \\ 
\hline 
\theta_f(t) & \dfrac{f'(t)}{\sqrt{1+f'^2(t)}}=g(t) &
\dfrac{1}{\sqrt{1+f'^2(t)}}=\sqrt{1-g^2(t)} &
f'(t)=\dfrac{g(t)}{\sqrt{1-g^2(t)}} \\ 
\hline 
\theta_e(t) & \dfrac{y'(t)}{\sqrt{x'^2(t)+y'^2(t)}} &
\dfrac{x'(t)}{\sqrt{x'^2(t)+y'^2(t)}} &
\dfrac{y'(t)}{x'(t)} \\ 
\hline 
\end{tblr}$$

\vspace{0.5cm}

We are going to present the compatibility condition \eqref{compcirc} in terms of polar and inclination angles. Let us denote by $\rho_f(t)$ and $\rho_e(t)$ the polar angles of the graph point $(t,f(t))$ and the equidistant point $(x(t),y(t))$, respectively. Also, let us denote by $\theta_f(t)$ and $\theta_e(t)$ the inclination angles of the tangent lines of the graph of $f$ and the equidistant function at the parameter $t$, respectively. In terms of these functions, \eqref{compcirc} can be written as
$$ \big[\sin(\theta_f(t))-\cos(\rho_e(t))\big]x'(t)=\big[\cos(\theta_f(t))+\sin(\rho_e(t))\big]y'(t).$$
Using Remark \ref{compcoeffpos} and Theorem \ref{derivnonzero:x}, 
$$ \tan(\theta_e(t))=\dfrac{y'(t)}{x'(t)} =  \dfrac{\sin(\theta_f(t))-\cos(\rho_e(t))}{\cos(\theta_f(t))+\sin(\rho_e(t))}. $$
As the sum-to-product identities show,
$$ \dfrac{\sin(\alpha)-\cos(\beta)}{\cos(\alpha)+\sin(\beta)}=
\dfrac{\sin(\alpha)-\sin(\frac{\pi}{2}-\beta)}{\cos(\alpha)+\cos(\frac{\pi}{2}-\beta)} = \tan\left(\frac{\alpha+\beta}{2}-\frac{\pi}{4} \right),$$
and we have that
$$ \tan(\theta_e(t)) = \tan\left(\frac{\theta_f(t)+\rho_e(t)}{2}-\frac{\pi}{4} \right) \ \ \Rightarrow \ \ \theta_e(t) = \frac{\theta_f(t)+\rho_e(t)}{2}-\frac{\pi}{4} \ \pmod{\pi}.  $$

\subsection{The parametric expressions of the equidistant points in higher dimensional spaces}

In what follows we are going to summarize the higher dimensional versions of the most important results as we have seen above in case of dimension two. Let $K\subset \mathbb{R}^{n+1}$ be a sphere centered at the origin of the Euclidean space and consider the focal set $L$ as the epigraph of a positive-valued, continuously differentiable convex function $f\colon \mathbb{R}^n\to \mathbb{R}^+$ such that
$$|t|^2+f^2(t) > R^2 \quad (t\in \mathbb{R}^n),$$
where $R$ is the radius of the sphere. Using the outer unit normal 
$$N=\frac{1}{\sqrt{1+|\nabla f|^2}} (\nabla f , -1)$$
to the graph of the function $f$, any equidistant point $(x, y)$ can be given in the parametric form
$$x(t)= t + \frac{s(t) }{\sqrt{1+|\nabla f |^2(t)}} \nabla f(t)\ \ \textrm{and}\ \ y(t) = f(t) - \frac{s(t)}{\sqrt{1+|\nabla f|^2(t)}},$$ 
where $(t,f(t))$ is the closest point of the epigraph to $(x,y)$ and
$$s(t)=\sqrt{|x(t)|^2+y^2(t)} -R.$$
Therefore
$$(s(t)+R)^2=|x(t)|^2+y^2(t)= \left |t + \frac{s(t) }{\sqrt{1+|\nabla f |^2(t)}} \nabla f(t)\right |^2+\left(f(t) - \frac{s(t)}{\sqrt{1+|\nabla f|^2 (t)}}\right)^2$$
and, consequently,
$$s(t)=\frac{1}{2}\frac{|t|^2+f^2(t)-R^2}{R-\alpha(t)}$$
provided that 
\begin{equation}
\label{condsphere:01} 
\alpha(t):=\frac{\langle t, \nabla f (t)\rangle-f(t)}{\sqrt{1+|\nabla f |^2 (t)}}< R,
\end{equation}
where $\alpha$ measures the signed distance of the tangent hyperplane to $f$ from the origin\footnote{The signed distance is negative if the origin is under the tangent hyperplane of the function $f$ and positive if the origin is over the tangent hyperplane of $f$.}. The parametric form of the equidistant points are
\begin{equation} 
\label{eqsphere:x}
 x(t)=t + \frac{1}{2\sqrt{1+| \nabla f|^2 (t)}}\cdot \frac{|t|^2+f^2(t)-R^2}{R-\alpha(t)}\nabla f (t),
\end{equation}
and
\begin{equation} 
\label{eqsphere:y}
y(t)=f(t) - \frac{1}{2\sqrt{1+|\nabla f|^2(t)}} \cdot \frac{|t|^2+f^2(t)-R^2}{R-\alpha(t)}.
\end{equation}

\subsubsection{Critical parameters}
The domain $D\subset \mathbb{R}^n$ for the parameters of the expressions \eqref{eqsphere:x} and \eqref{eqsphere:y} obviously contains the origin together with an open neighbourhood because of
$$\alpha(0)=\frac{-f(0)}{\sqrt{1+|\nabla f |^2 (0)}}< R.$$
Let $S_{n-1}\subset \mathbb{R}^n$ be the Euclidean unit sphere centered at the origin, $u\in S_{n-1}$, and define the slit (real) function $\varphi(r):=f(ru)$ where $r\geq 0$. We are going to compare the set $D$ of admissible parameters (for $f$) to the sets $D_u$ of admissible parameters for $\varphi$. Applying the reasoning from the 1-dimensional case for the real function $\varphi$, there exists a positive critical parameter $t_u^+$ along the ray $\{ru \mid r\geq 0\}$ such that
$$\alpha_u(r):=\frac{r\varphi'(r)-\varphi(r)}{\sqrt{1+\varphi'^2(r)}}< R \qquad (0\leq r< t_u^+),$$
so the admissible parameters for $\varphi$ form the segment (ray) $D_u=\{ru \mid 0\leq r < t_u^+\}$.

On the other hand, writing $t=ru$ (where $0\leq r < t_u^+$), because of
$$ \varphi'(r)=\langle \nabla f (ru), u\rangle,  \quad \langle t, \nabla f (t)\rangle = \langle ru, \nabla f (ru)\rangle=r \varphi'(r),$$
and
$$ \varphi'^2(r)=\langle \nabla f (ru), u\rangle^2\leq |\nabla f |^2(ru),$$
we can see that\footnote{Since $\alpha_u(r)$ is the signed distance of the tangent line of $\varphi$ from the origin, and $\alpha(t)$ is the signed distance of the tangent hyperplane of $f$ from the origin (at the same parameter $t=ru$), this inequality can also be verified geometrically.} 
\begin{equation}
\label{condsphere:02} 
\abs{\alpha(t)}=\frac{\abs{\langle t, \nabla f (t)\rangle-f(t)}}{\sqrt{1+|\nabla f |^2 (t)}} =
\frac{\abs{r\varphi'(r)-\varphi(r)}}{\sqrt{1+|\nabla f |^2 (t)}} \leq
\frac{\abs{r\varphi'(r)-\varphi(r)}}{\sqrt{1+\varphi'^2(r)}} = \abs{\alpha_u(r)}
\end{equation}
and, since the numerators are the same, $ \sgn \alpha(t) = \sgn \alpha_u(r)$. Therefore, if $\alpha_u(r)<R$, then
$$ \text{either} \ \ \alpha_u(r)\leq \alpha(t) \leq 0 <R  \ \ \text{or} \ \
0 \leq \alpha(t) \leq \alpha_u(r)<R, $$
so $\alpha(t)<R$ is also implied, i.e. parameters admissible for $\varphi$ are also admissible for $f$. Thus we have a star-shaped subdomain 
$$\bigcup_{u\in S_{n-1}} D_u \subset D.$$
In general, this subdomain can be strictly smaller than $D$, because the intersections of $D$ with central rays might continue after the critical parameters $t_u^+$ (they can even be disconnected), as the next example shows.

\begin{example} {\emph{Consider the function $f(t_1, t_2)= t_1^2+20t_1t_2+1000t_2^2-30t_2+1$. It is strictly convex because the kernel matrix has minors $\Delta_1=1$ and $\Delta_2=1000-10^2=900$, and there is a local minimum at $(t_1,t_2)=(-1/6,1/60)$ with minimum value $3/4$, which is also a global minimum, so $f$ is positive everywhere. For this function $f$,
$$\nabla f (t_1,t_2)=(2t_1+20t_2, \, 20t_1+2000t_2-30),$$
$$\alpha(t_1,t_2)= \dfrac{t_1^2+20t_1t_2+1000t_2^2-1}{\sqrt{1+(2t_1+20t_2)^2+(20t_1+2000t_2-30)^2}}.$$
Choosing the direction vector $u=(1,0)$, along this ray, we have the slit function
$$ \varphi(r)=f(r,0)=r^2+1 \quad \Rightarrow \quad \varphi'(r)=2r 
 \quad \Rightarrow \quad \alpha_u(r)=\dfrac{r^2-1}{\sqrt{1+4r^2}};$$
setting $R=0.3$, $D_u=\mathopen[0,t_u^+\mathclose[=\mathopen[0,1.37\mathclose[$. On the other hand,
$$\alpha(r,0)= \dfrac{r^2-1}{\sqrt{1+(2r)^2+(20r-30)^2}}=\dfrac{r^2-1}{\sqrt{404r^2-1200r+901}},$$
and the intersection of $D$ with the given ray consists of multiple disconnected segments, as shown in Figure \ref{counterex2d}.}}
\end{example}

\begin{figure}[h!] 
\includegraphics[width=.8\linewidth]{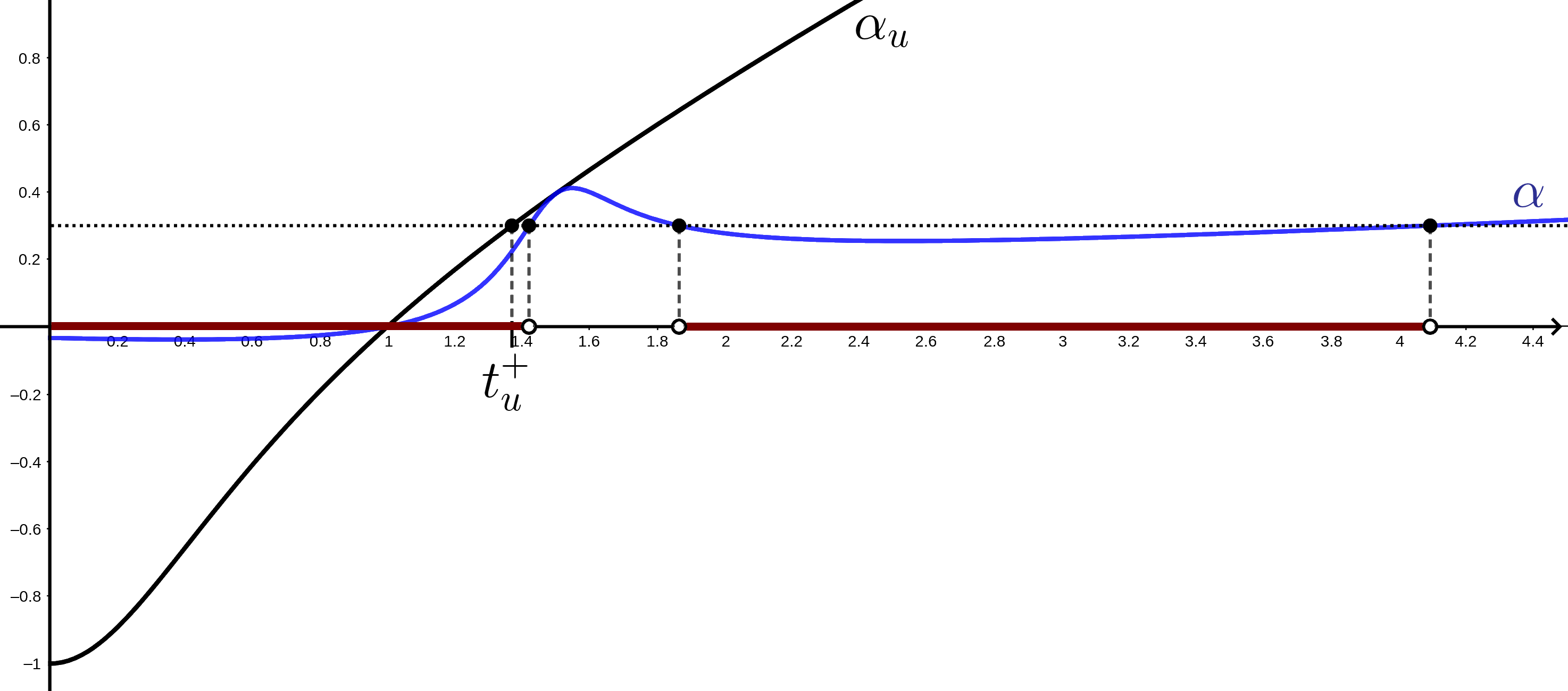}
\caption{\label{counterex2d} The signed distance functions $\alpha$ and $\alpha_u$ in case of $f(t_1, t_2)=t_1^2+20t_1t_2+1000t_2^2-30t_2+1$ with $u=(1,0)$.}
\end{figure}

\subsubsection{The expression of the mappings $f$ and $\nabla f$ in terms of the equidistant parameterization \eqref{eqsphere:x} and \eqref{eqsphere:y}}

Introducing the auxiliary functions
$$r(t)=\sqrt{|x(t)|^2+y^2(t)}>R\ \ \textrm{and}\ \ g(t)=\frac{x(t)-t}{r(t)-R},$$
we can repeat the calculations similarly to the two-dimensional case to conclude that  
$$g(t)=\frac{\nabla f}{\sqrt{1+|\nabla f|^2}}.$$
In particular, $|g(t)|<1$ and
$$\sqrt{1-|g(t)|^2}=\frac{1}{\sqrt{1+|\nabla f|^2}}.$$
Therefore
$$y(t)=f(t)-(r(t)-R)\sqrt{1-|g(t)|^2}$$
and we have
\begin{equation}
\label{compeq:higher01}
f(t)=y(t)+(r(t)-R)\sqrt{1-|g(t)|^2}.
\end{equation}
On the other hand
\begin{equation}
\label{compeq:higher02}
\nabla f(t)=\frac{g(t)}{\sqrt{1-|g(t)|^2}}.
\end{equation}

\subsubsection{The case of a twice continuously differentiable function: the compatibility of the equidistant parameterization} Let $f$ be a twice continuously differentiable (positive-valued convex) function. In this case the equidistant parameterization consists of continuously differentiable functions. Using the derivatives
$$\partial_i f(t)=\partial_i y(t)+\partial_i r(t)\sqrt{1-|g(t)|^2}-\left(r(t)-R\right)\frac{\langle g(t), \partial_i g(t) \rangle}{\sqrt{1-|g(t)|^2}},$$
$$\partial_i g(t)=\frac{(\partial_i x(t)-e^i)(r(t)-R)-(x(t)-t)\partial_i r(t)}{(r(t)-R)^2}=\frac{\partial_i x(t)-e^i-g(t)\partial_i r(t)}{r(t)-R},$$
$$\partial_i r(t)= \frac{\langle x(t), \partial_i x(t) \rangle +y(t)\partial_i y(t)}{r(t)}=\left\langle \frac{x(t)}{r(t)}, \partial_i x(t) \right\rangle +\frac{y(t)}{r(t)}\partial_i y(t), $$
the compatibility condition for the equidistant parameterization, obtained by comparing the right-hand sides of formulas \eqref{compeq:higher01} and \eqref{compeq:higher02}, is
\begin{align*}
\frac{g^i(t)}{\sqrt{1-|g(t)|^2}} &= \partial_i y(t)+\partial_i r(t)\sqrt{1-|g(t)|^2}-\left(r(t)-R\right)\frac{\langle g(t), \partial_i g(t) \rangle}{\sqrt{1-|g(t)|^2}}, \\
g^i(t) &= \sqrt{1-|g(t)|^2} \, \partial_i y(t)+(1-|g(t)|^2) \, \partial_i r(t)- \left(r(t)-R\right) \langle g(t), \partial_i g(t) \rangle,  \\
g^i(t) &= \sqrt{1-|g(t)|^2} \, \partial_i y(t)+(1-|g(t)|^2) \, \partial_i r(t)- \langle g(t), \partial_i x(t)-e^i-g(t)\partial_i r(t) \rangle,  \\
0      &= \sqrt{1-|g(t)|^2} \, \partial_i y(t)+ \partial_i r(t)- \langle g(t), \partial_i x(t) \rangle,  \\
0      &= \sqrt{1-|g(t)|^2} \, \partial_i y(t)+ \left\langle \frac{x(t)}{r(t)}, \partial_i x(t) \right\rangle +\frac{y(t)}{r(t)}\partial_i y(t) - \langle g(t), \partial_i x(t) \rangle,  \\
0      &= \left(\sqrt{1-|g(t)|^2}+\frac{y(t)}{r(t)}\right)\partial_i y(t) - \left \langle g(t)-\frac{x(t)}{r(t)}, \partial_i x(t)\right \rangle,
\end{align*}
so the compatibility condition is          
\begin{equation}
\label{compeq:higher03}
\left \langle g(t)-\frac{x(t)}{r(t)}, \partial_i x (t)\right \rangle=\left(\sqrt{1-|g(t)|^2}+\frac{y(t)}{r(t)}\right)\partial_i y(t).
\end{equation}

\begin{remark} \label{compcoeffpos:3d} {\emph{In \eqref{compeq:higher03}, the coefficient of $\partial_i y(t)$ on the right-hand side is always strictly positive similarly to the two-dimensional case; see Remark \ref{compcoeffpos}.}}
\end{remark}

\begin{theorem} 
\label{onetoone:3d}
The function $x\colon D\subset \mathbb{R}^n\to \mathbb{R}^n$,
$$ x(t)=t + \frac{1}{2\sqrt{1+| \nabla f|^2 (t)}} \cdot \frac{|t|^2+f^2(t)-R^2}{R-\alpha(t)}\nabla f (t)$$
is a one-to-one correspondence with nowhere vanishing Jacobian.
\end{theorem}

\begin{proof} Since for any $x\in \mathbb{R}^n$ we have a uniquely determined equidistant point $(x,y)$ together with a uniquely determined closest point $(t,f(t))$, where the distance of $(x,y)$ to the epigraph is attained at, the function $x\colon D\subset \mathbb{R}^n\to \mathbb{R}^n$ is a one-to-one correspondence. If $\partial_v x (t)=0$ then, by the compatibility condition \eqref{compeq:higher03}, it follows that $\partial_v y(t)=0$. Since the parametric expressions are of the form
$$x(t)=t+\varphi(t) \nabla f (t)\ \ \textrm{and}\ \ y(t)=f(t)-\varphi(t),$$
we have that 
$$0=v+\partial_v \varphi (t)\nabla f (t)+\partial_v \nabla f (t)\ \ \textrm{and}\ \ 0=\partial_v f (t)-\partial_v \varphi (t).$$
Therefore
$$0=v+\partial_v f(t)\nabla f (t)+\partial_v \nabla f (t)$$
and, taking the inner product with $v$,
$$0=|v|^2+\left(\partial_v f\right)^2(t)+f''(t)(v,v)\ \ \Rightarrow \ \ v=0,$$
as was to be proved. 
\end{proof}

\subsubsection{Necessary and sufficient conditions for a pair $x(t)$ and $y(t)$ to be the equidistant parameterization of the graph of the equidistant function belonging to $K$ and $L$}

\begin{theorem}
\label{equidpar3d:thm}
Let $x\colon D\subset \mathbb{R}^n\to \mathbb{R}^n$ and $y\colon D\subset \mathbb{R}^n\to \mathbb{R}$ be continuously differentiable functions defined on an open subset $D\subset \mathbb{R}^n$ containing the origin, and suppose that 
\begin{itemize}
\item [(i)] the polar distance is greater than the radius of the circle $K$ centered at the origin, that is
$$r(t):=\sqrt{|x(t)|^2+y^2(t)} > R,$$
\item[(ii)] the norm of the auxiliary function
$$g(t):=\frac{x(t)-t}{r(t)-R}$$
is less than one,
$$1>|g(t)|^2+\frac{1-\sgn y(t)}{2} \frac{y^2(t)}{\left (r(t)-R\right)^2},$$
\item [(iii)] the mapping $$\frac{g(t)}{\sqrt{1-g^2(t)}}$$ is monotone increasing and the compatibility condition \eqref{compeq:higher03} is satisfied.
\end{itemize}
Then the functions $x$ and $y$ give the equidistant parameterization of the graph of the equidistant function belonging to $f\colon \mathbb{R}^n\to \mathbb{R}$, where
\begin{equation}
\label{function3d:thm}
f(t)=y(t)+\left(r(t)-R\right)\sqrt{1-|g(t)|^{2}}
\end{equation}
for any $t\in D$. If $x\colon D\subset \mathbb{R}^n\to \mathbb{R}^n$ is a one-to-one correspondence with nowhere vanishing Jacobian, then its domain is maximal for the equidistant parameters. 
\end{theorem}

The proof can be given similarly to the two-dimensional case.

\begin{exercise} Prove that the equidistant function is the envelope of the upper branches of the one-parameter family of hyperbolas one of whose focal points runs along the function $f$ and the other one is uniformly given at the origin.  
\end{exercise}

Hint. The proof can be given similarly to the two-dimensional case.

\subsubsection{The characterization of equidistant functions} Let $f$ be a twice continuously differentiable, positive-valued, convex function.
In what follows we are looking for expressions of the equidistant parameterization in terms of the equidistant function $G=y\circ x^{-1}$. The steps are essentially similar to the two-dimensional case. Since $G(x(t))=y(t)$, it follows that 
$$\langle \nabla G(x(t)), \partial_i x(t)\rangle=\partial_i y(t)$$
and the compatibility condition \eqref{compeq:higher03} together with Theorem \ref{onetoone:3d} shows that
$$g(t)-x_0(t)=\left(\sqrt{1-|g(t)|^2}+y_0(t)\right)\nabla G(x(t)),$$
where
$$x_0(t)=\frac{x(t)}{r(t)}, \ \ y_0(t)=\frac{y(t)}{r(t)}\ \ \textrm{and}\ \ \sqrt{1-g^2(t)}+y_0(t)>0$$
because of $f(t)>0$. Therefore 
\begin{equation}
\label{keyeq:01}
g(t)=\sqrt{1-|g(t)|^2}\nabla G(x(t))+x_0(t)+y_0(t)\nabla G(x(t)).
\end{equation}
Introducing the abbreviation
$$\lambda(t)=\sqrt{1-|g(t)|^2} \ \ \Rightarrow \ \ |g(t)|^2=1-\lambda^2(t),$$
it follows, by taking the norm square of both sides of equation \eqref{keyeq:01}, that
$$|g(t)|^2=(1-|g(t)|^2)|\nabla G|^2(x(t))+2\sqrt{1-|g(t)|^2}\langle \nabla G(x(t)), x_0(t)+y_0(t)\nabla G(x(t))\rangle +$$
$$|x_0(t)+y_0(t)\nabla G(x(t))|^2,$$
that is
$$1-\lambda^2=\lambda^2|\nabla G|^2(x)+2\lambda\langle \nabla G(x), x_0+y_0\nabla G(x)\rangle +|x_0+y_0\nabla G(x)|^2$$
and we have a quadratic equation
$$0=\lambda^2\left(1+|\nabla G|^2(x)\right)+2\lambda\langle \nabla G(x), x_0+y_0\nabla G(x)\rangle +|x_0+y_0\nabla G(x)|^2-1.$$
Using that $x_0^2+y_0^2=1$, a straightforward calculation shows that its discriminant is
$$4\left(y_0-\langle x_0, \nabla G(x) \rangle \right)^2$$
and, consequently,
$$\lambda_1(t)=\frac{y_0(t)(1-|\nabla G|^2(x(t)))-2\langle x_0(t), \nabla G(x(t))\rangle}{1+|\nabla G|^2(x(t))} \ \ \textrm{or} \ \ \lambda_2(t)=-y_0(t).$$
It is easy to see that $\lambda(t)=-y_0(t)$ implies that
$$\sqrt{1-|g(t)|^2}+y_0(t)=0,$$
which is impossible in the sense of Remark \ref{compcoeffpos:3d}. Thus we have that 
\begin{equation}
\label{equidfunction3d:01}
\sqrt{1-|g(t)|^2}=\frac{y_0(t)(1-|\nabla G|^2(x(t)))-2\langle x_0(t), \nabla G(x(t))\rangle}{1+|\nabla G|^2(x(t))}.
\end{equation}
Therefore, by \eqref{keyeq:01},
$$g(t)=\frac{y_0(t)(1-|\nabla G|^2(x(t))-2\langle x_0(t), \nabla G(x(t))\rangle}{1+|\nabla G|^2(x(t))}\nabla G (x(t))+x_0(t)+y_0(t)\nabla G(x(t)).$$
In a more compact form,
\begin{equation}
\label{equidfunction3d:02}
g(t)=x_0(t)+2\frac{y_0(t)-\langle x_0(t), \nabla G(x(t))\rangle}{1+|\nabla G|^2(x(t))}\nabla G (x(t)).
\end{equation}
Since
$$g(t)=\frac{x(t)-t}{r(t)-R},$$
$$t=x(t)-\frac{r(t)-R}{r(t)}\left(x(t)+2\frac{y(t)-\langle x(t), \nabla G(x(t))\rangle}{1+|\nabla G|^2(x(t))}\nabla G (x(t))\right),$$
where
$$y(t)=G(x(t))\ \ \textrm{and} \ \ r(t)=\sqrt{|x(t)|^2+G^2(x(t))}.$$
Introducing the function
\begin{equation}
\label{equidfunction:03}
h(x):=x-\frac{\sqrt{|x|^2+G^2(x)}-R}{\sqrt{|x|^2+G^2(x)}}\left(x+2\frac{G(x)-\langle x, \nabla G(x)\rangle}{1+|\nabla G|^2(x)}\nabla G (x)\right),
\end{equation}
it follows that $h(x(t))=t$. Since the function \eqref{equidfunction:03}
depends only on the equidistant function we can formulate the following result as an application of Theorem \ref{equidpar3d:thm}.

\begin{theorem}
Let $G\colon \mathbb{R}^n\to \mathbb{R}$ be a twice continuously differentiable function and suppose that the function $h\colon \mathbb{R}^n\to D\subset \mathbb{R}^n$ defined by formula \eqref{equidfunction:03} is a one-to-one correspondence with nowhere vanishing derivative and its range is an open set containing the origin. The function $G$ is an equidistant function if and only if the pair of functions $x(t):=h^{-1}(t)$ and $y(t):=G(x(t))$ is the equidistant parameterization for the graph of $G$. 
\end{theorem}

\begin{remark} {\emph{As we have seen above, the focal sets are the sphere centered at the origin with radius $R$ such that $\sqrt{|x|^2+G^2(x)}>R$ and the epigraph of the positive-valued, twice continuously differentiable, convex function $f\colon \mathbb{R}^n\to \mathbb{R}$ defined by formula \eqref{function3d:thm}.}}
\end{remark} 

\section{Acknowledgement}

Myroslav Stoika is supported by the Visegrad Scholarship Program. M\'ark Ol\'ah has received funding from the HUN-REN Hungarian Research Network.

\end{document}